\documentclass[paper=letterpaper, onecolumn, fontsize=11pt]{article}

\usepackage[english]{babel} 
\usepackage{amsmath,amsfonts,amssymb} 
\usepackage{latexsym,amscd,verbatim,alltt,array}
\usepackage{euscript}
\usepackage[affil-it]{authblk}
\usepackage{tikz}
\usepackage{amsmath}
\usepackage{amsfonts}
\usepackage{amsmath}
\numberwithin{equation}{section}
\usepackage{enumitem}
\usetikzlibrary{decorations.pathreplacing,decorations.markings}
\usepackage{multicol}

\makeatletter  
\long\def\@makecaption#1#2{%
  \vskip\abovecaptionskip
  \sbox\@tempboxa{#1 #2}%
  \ifdim \wd\@tempboxa >\hsize
    #1 #2\par
  \else
    \global \@minipagefalse
    \hb@xt@\hsize{\hfil\box\@tempboxa\hfil}%
  \fi
  \vskip\belowcaptionskip}
\makeatother 

\tikzset{ 
  on each segment/.style={
    decorate,
    decoration={
      show path construction,
      moveto code={},
      lineto code={
        \path [#1]
        (\tikzinputsegmentfirst) -- (\tikzinputsegmentlast);
      },
      curveto code={
        \path [#1] (\tikzinputsegmentfirst)
        .. controls
        (\tikzinputsegmentsupporta) and (\tikzinputsegmentsupportb)
        ..
        (\tikzinputsegmentlast);
      },
      closepath code={
        \path [#1]
        (\tikzinputsegmentfirst) -- (\tikzinputsegmentlast);
      },
    },
  },
  mid arrow/.style={postaction={decorate,decoration={
        markings,
        mark=at position .6 with {\arrow[#1]{stealth}}
      }}},
}

\newtheorem{theorem}{Theorem}[section]
\newtheorem{lemma}[theorem]{Lemma}
\newtheorem{proposition}[theorem]{Proposition}
\newtheorem{corollary}[theorem]{Corollary}

\newtheorem{definition}[theorem]{Definition}
\newtheorem{remark}[theorem]{Remark}
\newtheorem{example}[theorem]{Example}

\newcommand{\qed}{{\unskip\nobreak\hfil\penalty50\hskip1em\hbox{}\nobreak
   \hfil \ensuremath{\Box}\parfillskip=0pt \par}}

\title{\vspace{-2.5cm} Unmixedness of some weighted oriented graphs}
\author{\small Lourdes Cruz\thanks{\texttt{lcruzg@math.cinvestav.mx}}\ }
\author{\small Yuriko Pitones\thanks{\texttt{ypitones@math.cinvestav.mx}}\ }
\author{\small Enrique Reyes\thanks{\texttt{ereyes@math.cinvestav.mx}}\ }
\affil{\vspace{-0.2cm}\scriptsize Departamento de Matem\'{a}ticas\\ Centro de Investigaci\'{o}n y de Estudios Avanzados del Instituto Polit\'ecnico Nacional\\ Apartado Postal 14--740, Ciudad de M\'{e}xico\\ 07000 M\'exico}
\date{\tiny\ }

\begin{document}
\maketitle
\thispagestyle{empty}

\vspace{-10ex}

\parindent=8mm

\begin{abstract}
\noindent
Let $D=(G,\mathcal{O},w)$ be a weighted oriented graph whose edge ideal is $I(D)$. In this paper, we characterize the unmixed property of $I(D)$ for each one of the following cases: $G$ is an $SCQ$ graph; $G$ is a chordal graph; $G$ is a simplicial graph; $G$ is a perfect graph; $G$ has no $4$- or $5$-cycles; $G$ is a graph without $3$- and $5$-cycles; and ${\rm girth}(G)\geqslant 5$.
\end{abstract}

\noindent
{\small\textbf{Keywords:} Weighted oriented graphs, edge ideal, unmixed ideal, $SCQ$ graph, perfect graph.}

\section{Introduction}
A {\it weighted oriented graph\/} $D$ is a triplet $(G,\mathcal{O},w)$, where $G$ is a simple graph whose vertex set is $V(G)$; $\mathcal{O}$ is an edge orientation of $G$ (an assignment of a direction to each edge of $G$); and $w$ is a function $w:V(G) \to\mathbb{N}$. In this case, $G$ is called the {\it underlying graph\/} of $D$. The vertex set of $D$ is $V(D):=V(G)$, the edge set of $D$, denoted by $E(D)$, is the set of oriented edges of the oriented graph $(G,\mathcal{O})$. The \emph{weight} of $x\in V(D)$ is $w(x)$ and we denote the set $\{x\in V(D)\mid w(x)>1\}$ by $V^{+}$. If $R=K[x_{1},\ldots, x_{n}]$ is a polynomial ring over a field $K$, then the edge ideal of $D$ is $I(D)=(x_{i}x_{j}^{w(x_{j})}\mid (x_i,x_j)\in E(D))$ where $V(D)=\{ x_1,\ldots ,x_n\}$. These ideals (introduced in \cite{WOG}) generalize the usual edge ideals of graphs (see \cite{Villa}), since if $w(x)=1$ for each $x\in V(D)$, then $I(D)$ is the edge ideal of $G$, i.e. $I(D)=I(G)$. An interest in $I(D)$ comes from coding theory in some studies of Reed--Muller typed codes, (see \cite{Carvalho,pit-villa}). Furthermore, some algebraic and combinatorial invariants and proper\-ties of $I(D)$ have been studied in some papers \cite{vivares,reyes-villa,WOG,V-R-P,Zhu}. In particular, in \cite{WOG} is given a characterization of the irredundant irreducible decomposition of $I(D)$. This characterization permits studying when $I(D)$ is unmixed, using the strong vertex covers (Definition \ref{strong-cover-defn} and Theorem \ref{theorem42}). The unmixed property of $I(D)$ have been studied when $G$ is one of the following graphs: cycles in \cite{WOG}; graphs with whiskers in \cite{vivares,WOG}; bipartite graphs in \cite{reyes-villa,WOG}; graphs without $3$- $5$- and $7$-cycles and K\"onig graphs in \cite{V-R-P}.

\noindent
In this paper, we study the unmixed property of $I(D)$ for some families of weighted oriented graphs. In Section 2, we give the known results and definitions that we will use in the following sections. In Section 3 (in Theorem \ref{theorem-oct29}), we characterize when a subset of $V(D)$ is contained in a strong vertex cover. Using this result, we characterize the unmixed property of $I(D)$, when $G$ is a perfect graph (see Theorem \ref{Perf-Unm}). In Section 4 (in Theorem \ref{SCQ-char}), we characterize the unmixed property of $I(D)$ when $G$ is an $SCQ$ graph (see Definition \ref{CondSCQ}). These graphs generalize the graph defined in \cite{Rand} and in the context of this paper, they are important because if $G$ is well--covered such that $G$ is simplicial, $G$ is chordal or $G$ has no some small cycles, then $G$ is an $SCQ$ graph (see Remark \ref{rem-nov13} and Theorem \ref{wellcovered-characterization2}). In \cite{SCQ-ivan}, using the $SCQ$ graphs the authors characterize the vertex decomposable property of $G$ when each $5$-cycle of $G$ has at least $4$ chords. Also, in Section 4, we characterize the unmixed property of $I(D)$ when $G$ is K\"oning, $G$ is simplicial or $G$ is chordal (see Corollaries \ref{Koning-Unm} and \ref{Simp-Chor-Unm}). In Section 5, we characterize the unmixed property of $I(D)$, when $G$ has no $3$- or $5$-cycles; or $G$ has no $4$- or $5$-cycles; or $G$ has girth greater than 4 (see Theorems \ref{theorem-oct31}, \ref{No4,5Cyc-Unmix} and \ref{No3,4,5Cyc-Unmix}). Finally, in Section 6, we give some examples. Our results generalize the results about the unmixed property of $I(D)$ given in \cite{vivares,reyes-villa,WOG,V-R-P}, since if $G$ is well-covered and $G$ is one of the following graphs: cycles, graphs with whiskers, bipartite graphs, K\"oning graphs, or graphs without $3$-, $5$- and $7$-cycles, then $G$ is an $SCQ$ graph.

\section{Preliminaries}
In this Section, we give some definitions and well-known results that we will use in the following sections. Let $D=(G,\mathcal{O}, w)$ be a weighted oriented graph, recall that $V^{+}=\{ x\in V(D) \mid w(x)>1\}$ and $I(D)=\big( x_ix_{j}^{w(x_j)} \mid (x_i,x_j)\in E(D)\big)$.

\begin{definition}\rm
Let $x$ be a vertex of $D$, the sets 
$$N_{D}^{+}(x):=\{y\mid(x,y)\in E(D)\} \quad {\rm and} \quad N_{D}^{-}(x):=\{y\mid(y,x)\in E(D)\}$$ are called the {\it out-neighbourhood\/} and the {\it in-neighbourhood\/} of $x$, respectively. The {\it neighbourhood\/} of $x$ is the set $N_{D}(x):=N_{D}^{+}(x)\cup N_{D}^{-}(x)$. Furthermore, $N_D[x]:=N_D(x)\cup \{ x\}$. Also, if $A\subseteq V(D)$ then $N_D(A):=\{ b\in V(D)\mid b\in N_D(a) \ {\rm for \ some} \ a\in A\}$.
\end{definition}

\begin{definition}\rm\label{Sink-source}
Let $x$ be a vertex of $D$. If $N_{D}^{+}(x)=\emptyset$, then $x$ is called a {\it sink\/}. On the other hand, $x$ is a {\it source\/} if $N_{D}^{-}(x)=\emptyset$.
\end{definition}

\begin{remark}\rm\label{rem-V-R-P}
Consider the weighted oriented graph $\tilde{D}=(G,\mathcal{O},\tilde{w})$ with $\tilde{w}(x)=1$ if $x$ is a source and $\tilde{w}(x)=w(x)$ if $x$ is not a source. Hence, $I(\tilde{D})=I(D)$. Therefore, in this paper, we assume that if $x$ is a source, then $w(x)=1$. 
\end{remark}

\begin{definition}\rm
The {\it degree of $x\in V(D)$\/} is $deg_G(x):=|N_D(x)|$ and $N_G(x):=N_D(x)$. 	 
\end{definition}

\begin{definition}\rm
A {\it vertex cover\/} $\mathcal{C}$ of $D$ (resp. of $G$) is a subset of $V(D)$ (resp. of $V(G)$), such that if $(x,y)\in E(D)$ (resp. $\{ x,y\} \in E(G)$), then $x\in \mathcal{C}$ or $y\in \mathcal{C}$. A vertex cover $\mathcal{C}$ of $D$ is {\it minimal\/} if each proper subset of $\mathcal{C}$ is not a vertex cover of $D$.
\end{definition}

\begin{remark}\rm\label{Compl-einC}
Let $\mathcal{C}$ be a vertex cover of $D$ and $e\in E(G)$. Then, $\mathcal{C} \cap e\neq \emptyset$. Furthermore, $e\cap (\mathcal{C}\setminus a)\neq \emptyset$ if $a\notin e$, $b\in N_D(a)$ and $e=\{ a,b\}$. Hence, $(\mathcal{C} \setminus a)\cup N_D(a)$ is a vertex cover of $D$.
\end{remark}

\begin{definition}\rm\label{L-sets}
Let $\mathcal{C}$ be a vertex cover of $D$, we define the following three sets: 
\begin{itemize}[noitemsep]
\item $L_1(\mathcal{C}):=\{x\in \mathcal{C}\mid N_{D}^{+}(x)\cap \mathcal{C}^{c}\neq \emptyset \}$ where $\mathcal{C}^{c}=V(D)\setminus \mathcal{C}$,
\item $L_2(\mathcal{C}):=\{x\in \mathcal{C}\mid\mbox{$x\notin L_1(\mathcal{C})$ and $N^{-}_{D}(x)\cap \mathcal{C}^c\neq\emptyset$}\}$,
\item $L_3(\mathcal{C}):=\mathcal{C}\setminus(L_1(\mathcal{C})\cup L_2(\mathcal{C}))$.
\end{itemize}
\end{definition}

\begin{remark}\rm\label{einC}
If $\mathcal{C}$ is a vertex cover of $G$, $x\in V(G)\setminus \mathcal{C}$ and $y\in N_G(x)$, then $e:=\{ x,y\} \in E(G)$ and $e\cap \mathcal{C} \neq \emptyset$. So, $y\in \mathcal{C}$, since $x\notin \mathcal{C}$. Hence, $N_G(x)\subseteq \mathcal{C}$.
\end{remark}

\begin{remark}\rm\label{VertexL3}
Let $\mathcal{C}$ be a vertex cover of $D$, then $x\in L_3(\mathcal{C})$ if and only if $N_{D}[x]\subseteq \mathcal{C}$. Hence, $L_3(\mathcal{C})= \emptyset$ if and only if $\mathcal{C}$ is minimal.
\end{remark}

\begin{definition}\rm\label{strong-cover-defn}
A vertex cover $\mathcal{C}$ of $D$ is {\it strong\/} if for each $x\in L_3(\mathcal{C})$ there is $(y,x)\in E(D)$ such that $y\in L_2(\mathcal{C})\cup L_{3}(\mathcal{C})=\mathcal{C} \setminus L_1(\mathcal{C})$ with $y\in V^{+}$ (i.e. $w(y)>1$).
\end{definition}

\begin{definition}\rm
An ideal $I$ of a ring $R$ is {\it unmixed\/} if each one of its associated primes has the same height. 
\end{definition}
 
\begin{theorem}\label{theorem42}{\rm \cite[Theorem 31]{WOG}}
The following conditions are equivalent:
\begin{enumerate}[noitemsep]
\item[{\rm (1)}] $I(D)$ is unmixed.
\item[{\rm (2)}] Each strong vertex cover of $D$ has the same cardinality. 
\item[{\rm (3)}] $I(G)$ is unmixed and $L_{3}(\mathcal{C})=\emptyset$ for each strong vertex cover $\mathcal{C}$ of $D$.
\end{enumerate}
\end{theorem}

\begin{definition}\rm
The {\it cover number\/} of $G$ is $\tau(G):={\rm min \ }\{ |\mathcal{C}|\ \mid \mathcal{C} \ {\rm \ is\ a\ vertex\ cover\ of\ } G\}$. Furthermore, a {\it $\tau$-reduction\/} of $G$ is a collection of pairwise disjoint induced subgraphs $H_1,\ldots ,H_s$ of $G$ such that $V(G)=\cup_{i=1}^{s} V(H_i)$ and $\tau(G)=\sum_{i=1}^{s} \tau(H_i)$.
\end{definition}

\begin{remark}\rm\label{MinimalStrongProp}
We have $\tau (G)=|\mathcal{C}_1|$, for some vertex cover $\mathcal{C}_1$. So, $\mathcal{C}_1$ is minimal. Thus, by Remark \ref{VertexL3}, $L_3(\mathcal{C}_1)=\emptyset$. Hence, $\mathcal{C}_1$ is strong. Now, if $I(D)$ is unmixed, then by {\rm (2)} in Theorem \ref{theorem42}, $|\mathcal{C}|=|\mathcal{C}_1|=\tau (G)$ for each strong vertex cover $\mathcal{C}$ of $D$.
\end{remark}

\begin{definition}\rm
A {\it stable set\/} of $G$ is a subset of $V(G)$ containing no edge of $G$. The {\it stable number of $G$\/}, denoted by $\beta(G)$, is $\beta(G):={\rm max \ } \{ |S|\ \mid S {\rm \ is\ a\ stable\ set\ of\ }G\}$. Furthermore $G$ is {\it well--covered\/} if $|S|=\beta(G)$ for each maximal stable set $S$ of $G$.
\end{definition}

\begin{remark}\rm\label{tau-beta}
$S$ is a stable set of $G$ if and only if $V(G)\setminus S$ is a vertex cover. Hence, $\tau (G)=|V(G)|-\beta (G)$.
\end{remark}

\begin{remark}\rm\label{1star}{\rm \cite[Remark 2.12]{V-R-P}}
$G$ is well-covered if and only if $I(G)$ is unmixed.
\end{remark}

\begin{definition}\rm
A collection of pairwise disjoint edges of $G$ is called a {\it matching\/}. A {\it perfect matching\/} is a matching whose union is $V(G)$. On the other hand, $G$ is a {\it K\"{o}nig graph\/} if $\tau(G) =\nu(G)$ where $\nu(G)$ is the maximum cardinality of a matching of $G$.
\end{definition}

\begin{definition}\rm\label{defi-VPR2.14}
Let $e$ be an edge of $G$. If $\{a,a^{\prime}\}\in E(G)$ for each pair of edges, $\{a,b\}$, $\{a^{\prime},b^{\prime}\}\in E(G)$ and $e=\{b,b^{\prime}\}$, then we say that $e$ {\it has the property \bf{(P)}\/}. On the other hand, we say that a matching $P$ of $G$ {\it has the property \bf{(P)}\/} if each edge of $P$ has the property \bf{(P)}.
\end{definition}

\begin{theorem}\label{Koning-Char}{\rm \cite[Proposition 15]{Ivan-Reyes}}
If $G$ is a K\"oning graph without isolated vertices, then $G$ is well--covered if and only if $G$ has a perfect matching with the property {\bf(P)}.
\end{theorem}

\begin{definition}\rm 
$\mathcal{P}=(x_1,\ldots ,x_n)$ is a {\it walk\/} (resp. an {\it oriented walk\/}) if $\{ x_i,x_{i+1}\} \in E(G)$ for $i=1,\ldots ,n-1$. In this case, $\mathcal{P}$ is a {\it path\/} (resp. an {\it oriented path\/}) if $x_1,\ldots ,x_n$ are different. On the other hand, a walk (resp. an oriented walk), $C=(z_1,z_2,\ldots ,z_n,z_1)$ is a {\it $n$-cycle\/} (resp. an {\it oriented $n$-cycle\/}) if $(z_1,\ldots ,z_n)$ is a path (resp. is an oriented path).
\end{definition}

\begin{definition}\rm
Let $A$ be a subset of $V(G)$, then the {\it graph induced by $A$\/}, denoted by $G[A]$, is the subgraph $G_1$ of $G$ with $V(G_1)=A$ and $E(G_1)=\{ e\in E(G)\mid e\subseteq A\}$. On the other hand, a subgraph $H$ of $G$ is {\it induced\/} if there is $B\subseteq V(G)$ such that $H=G[B]$.

\noindent
A cycle $C$ of $G$ is {\it induced\/} if $C$ is an induced subgraph of $G$.
\end{definition}

\begin{definition}\rm
A weighted oriented graph $D^{\prime}=(G^{\prime},\mathcal{O}^{\prime},w^{\prime})$ is a {\it weighted oriented subgraph of  $D=(G,\mathcal{O},w)$\/}, if $(G^{\prime},\mathcal{O}^{\prime})$ is an oriented subgraph of $(G,\mathcal{O})$  and $w^{\prime}(x)=w(x)$ for each $x\in V(G^{\prime})$. Furthermore, $D^{\prime}$ is an {\it induced weighted oriented subgraph of  $D$\/} if $G^{\prime}$ is an induced subgraph of $G$.
\end{definition}

\begin{definition}\rm
A vertex $v$ is called {\it simplicial} if the induced subgraph $H=G[N_{G}[v]]$ is a complete graph with $k=|V(H)|-1$, in this case, $H$ is called $k$-{\it simplex} (or {\it simplex}). The set of simplexes of $G$ is denoted by $S_G$. $G$ is a {\it simplicial graph\/} if every vertex of $G$ is a simplicial vertex of $G$ or is adjacent to a simplicial vertex of $G$. 
\end{definition}

\begin{definition}\rm
The minimum length of a cycle (contained) in a graph $G$, is called the {\it girth of $G$\/}. On the other hand, $G$ is a {\it chordal graph\/} if the induced cycles are $3$-cycles.
\end{definition}

\begin{theorem}{\rm \cite[Theorems 1 and 2]{Prisner}}\label{T1-T2Prisner}
If $G$ is a chordal or simplicial graph, then $G$ is well-covered if and only if every vertex of $G$ belongs to exactly one simplex of $G$. 
\end{theorem}

\begin{definition}\rm\label{CondSCQ}
An induced $5$-cycle $C$ of $G$ is called {\it basic} if $C$ does not contain two adjacent vertices of degree three or more in $G$. $G$ is an $SCQ$ graph (or $G\in SCQ$) if $G$ satisfies the following conditions:
\begin{enumerate}[noitemsep]
\item[$(i)$] There is $Q_G$ such that $Q_G=\emptyset$ or $Q_G$ is a matching of $G$ with the property {\bf{(P)}}.
\item[$(ii)$] $\{V(H)\mid H\in S_G\cup C_G\cup Q_G\}$ is a partition of $V(G)$, where $C_G$ is the set of basic $5$-cycles. 
\end{enumerate}
\end{definition}

\noindent
In the following three results, we use the graphs of Figure \ref{specialgraphs}.

\begin{figure}[h]
\centering
\begin{tikzpicture}[dot/.style={draw,fill,circle,inner sep=1pt},scale=.82] 

%%%%%%%%%%%%%%%%%%%%%%%%%%%%%%%%%%%%%%%%%

  \foreach \l [count=\n] in {{},{},{},{},{},{},{}} {
   \pgfmathsetmacro\angle{90-360/7*(\n-1)}
      \node[dot,label={\angle:$\l$}] (n\n) at (\angle:1) {};
  }
  \draw (n6) -- (n7) -- (n1) -- (n2) -- (n3) -- (n4)--(n5)--(n6);
  \node (0) at (0,-1.5){$\mathbf{C_{7}}$};

%%%%%%%%%%%%%%%%%%%%%%%%%%%%%%%%%%%%%%%%%%%%

\node[draw,fill,circle,inner sep=1pt] (2) at (-5.8,.9){};
\node (12) at (-5.6,1.1){{\small $d_{\tiny 1}$}};
\node[draw,fill,circle,inner sep=1pt] (5) at (-5.8,-.5) {};
\node (15) at (-5.9,-.8){{\small $d_{\tiny 2}$}};
\node[draw,fill,circle,inner sep=1pt](3) at (-4.4,.9) {};
\node (13) at (-4.1,.9){{\small $b_{\tiny 2}$}};
\node[draw,fill,circle,inner sep=1pt] (6) at (-4.4,-.5){};
\node (16) at (-4.1,-.5){{\small $a_{\tiny 2}$}};
\node[draw,fill,circle,inner sep=1pt] (1) at (-7.2,.9){};
\node (11) at (-7.5,.9){{\small $a_{\tiny 1}$}};
\node[draw,fill,circle,inner sep=1pt] (4) at (-7.2,-.5) {};
\node (14) at (-7.5,-.5){{\small $b_{\tiny 1}$}};
\node[draw,fill,circle,inner sep=1pt] (7) at (-3,2.3){};
\node (17) at (-2.7,2.3){{\small $c_{\tiny 1}$}};
\node[draw,fill,circle,inner sep=1pt] (8) at (-3,-1.9){};
\node (18) at (-2.7,-2.1){{\small $c_{\tiny 2}$}};
\node[draw,fill,circle,inner sep=1pt] (10) at (-5.1,-1.9){};
\node (20) at (-5.1,-2.2){{\small $g_{\tiny 2}$}};
\node[draw,fill,circle,inner sep=1pt] (9) at (-6.5,2.3){};
\node (19) at (-6.5,2.6){{\small $g_{\tiny 1}$}};

\node (0) at (-6.5,-2){$\mathbf{P_{10}}$};
\node (21) at (-6.5,0.4){\textcolor{blue}{$C_{1}$}};
\node (22) at (-5.1,0){\textcolor{blue}{$C_{2}$}};
\node (23) at (-7.5,0.2){{\small $\mathbf{\tilde{e}_{\tiny 1}}$}};
\node (24) at (-4.1,0.2){{\small $\mathbf{\tilde{e}_{\tiny 2}}$}};
\node (25) at (-2.7,0.2){{\small $\mathbf{\tilde{e}_{\tiny 3}}$}};

\draw[-] (1) -- (4)-- (5)--(2)--(3)--(6);
\draw[-] (1) -- (9)--(2);
\draw[-] (9) -- (7)--(8)--(10)--(6);
\draw[-] (5) -- (10);

%%%%%%%%%%%%%%%%%%%%%%%%%%%%%

\node[draw,fill,circle,inner sep=1pt] (2) at (5.8,.8){};
\node (12) at (6.05,1.05){{\small $b_{\tiny 4}$}};
\node[draw,fill,circle,inner sep=1pt] (5) at (5.8,-.6) {};
\node (15) at (6.05,-.85){{\small $b_{\tiny 3}$}};
\node[draw,fill,circle,inner sep=1pt](3) at (4.7,.8) {};
\node (13) at (4.45,1.05){{\small $b_{\tiny 2}$}};
\node[draw,fill,circle,inner sep=1pt] (6) at (4.7,-.6){};
\node (16) at (4.55,-.85){{\small $b_{\tiny 1}$}};
\node[draw,fill,circle,inner sep=1pt] (1) at (6.9,.8){};
\node (11) at (7.2,.8){{\small $a_{\tiny 4}$}};
\node[draw,fill,circle,inner sep=1pt] (4) at (6.9,-.6) {};
\node (14) at (7.2,-.6){{\small $a_{\tiny 3}$}};
\node[draw,fill,circle,inner sep=1pt] (7) at (3.6,.8){};
\node (17) at (3.3,.8){{\small $a_{\tiny 2}$}};
\node[draw,fill,circle,inner sep=1pt] (8) at (3.6,-.6){};
\node (18) at (3.3,-.6){{\small $a_{\tiny 1}$}};
\node[draw,fill,circle,inner sep=1pt] (10) at (5.25,2.05){};
\node (20) at (5.25,2.35){{\small $d_{\tiny 2}$}};
\node[draw,fill,circle,inner sep=1pt] (9) at (5.25,-1.85){};
\node (19) at (5.25,-2.15){{\small $d_{\tiny 1}$}};
\node[draw,fill,circle,inner sep=1pt] (11) at (4.7,0.1){};
\node (21) at (4.93,0.28){{\small $c_{\tiny 1}$}};
\node[draw,fill,circle,inner sep=1pt] (12) at (5.8,0.1){};
\node (22) at (5.57,0.28){{\small $c_{\tiny 2}$}};

\node[draw,fill,circle,inner sep=1pt] (13) at (2.42,0.1){};
\node (23) at (2.24,0.1){{\small $v$}};

\node (0) at (6,-1.9){$\mathbf{P_{13}}$};
\node (24) at (4.15,0.1){\textcolor{blue}{\small $C_{1}$}};
\node (25) at (6.35,0.1){\textcolor{blue}{\small $C_{2}$}};
\node (26) at (3.3,0.1){{\small $\mathbf{\tilde{e}_{\tiny 1}}$}};
\node (27) at (7.2,0.1){{\small $\mathbf{\tilde{e}_{\tiny 2}}$}};

\draw[-] (1) -- (2)-- (5)--(4)--(1);
\draw[-] (7) -- (8)--(6)--(3)--(7);
\draw[-] (9) -- (6);
\draw[-](5)--(9);
\draw[-] (10) -- (3);
\draw[-](10)--(2);
\draw[-] (11) -- (12);
\draw[-] (10) .. controls (1.46,0.1) .. (9);

%%%%%%%%%%%%%%%%%%%%%%%%%%%%%%%%%%%%%%%%%%%%%

\node[draw,fill,circle,inner sep=1pt] (1) at (0,-3){};
\node (11) at (0,-2.75){{\small $v$}};
\node[draw,fill,circle,inner sep=1pt] (2) at (-1,-4.3){};
\node (12) at (-1.3,-4.3){{\small $a_{\tiny 1}$}};
\node[draw,fill,circle,inner sep=1pt] (3) at (1,-4.3){};
\node (13) at (1.35,-4.3){{\small $a_{\tiny 3}$}};
\node[draw,fill,circle,inner sep=1pt] (4) at (0,-4.1){};
\node (14) at (0.35,-4.1){{\small $a_{\tiny 2}$}};
\node[draw,fill,circle,inner sep=1pt] (5) at (-1,-5.6){};
\node (15) at (-1.3,-5.6){{\small $b_{\tiny 1}$}};
\node[draw,fill,circle,inner sep=1pt] (6) at (-1,-6.9){};
\node (16) at (-1.3,-6.9){{\small $c_{\tiny 1}$}};
\node[draw,fill,circle,inner sep=1pt] (7) at (1,-5.6){};
\node (17) at (1.35,-5.6){{\small $b_{\tiny 3}$}};
\node[draw,fill,circle,inner sep=1pt] (8) at (1,-6.9){};
\node (18) at (1.35,-6.9){{\small $c_{\tiny 3}$}};
\node[draw,fill,circle,inner sep=1pt] (9) at (0,-5.2){};
\node (19) at (0.35,-5.2){{\small $b_{\tiny 2}$}};
\node[draw,fill,circle,inner sep=1pt] (10) at (0,-6.2){};
\node (20) at (0.3,-6.15){{\small $c_{\tiny 2}$}};
\node (0) at (0,-7.5){$\mathbf{T_{10}}$};

\node (21) at (-1.3,-4.95){{\small $\mathbf{\tilde{e}_{\tiny 1}}$}};
\node (22) at (0.35,-4.65){{\small $\mathbf{\tilde{e}_{\tiny 2}}$}};
\node (23) at (1.35,-4.95){{\small $\mathbf{\tilde{e}_{\tiny 3}}$}};

\draw[-] (1) -- (2)-- (5)--(6)--(10)--(9)--(4)--(1)--(3)--(7)--(8)--(6);
\draw[-] (10)--(8);

%%%%%%%%%%%%%%%%%%%%%%%%%%%%%%%%%%%%

\node[draw,fill,circle,inner sep=1pt] (1) at (-5,-3){};
\node (11) at (-5,-2.75){{\small $a_{\tiny 1}$}};
\node[draw,fill,circle,inner sep=1pt] (2) at (-3.5,-4){};
\node (12) at (-3.2,-3.95){{\small $a_{\tiny 2}$}};
\node[draw,fill,circle,inner sep=1pt] (7) at (-6.5,-4){};
\node (17) at (-6.8,-3.95){{\small $a_{\tiny 7}$}};
\node[draw,fill,circle,inner sep=1pt] (6) at (-7,-5.5){};
\node (16) at (-7.3,-5.5){{\small $a_{\tiny 6}$}};
\node[draw,fill,circle,inner sep=1pt] (3) at (-3,-5.5){};
\node (13) at (-2.65,-5.5){{\small $a_{\tiny 3}$}};
\node[draw,fill,circle,inner sep=1pt] (4) at (-3.7,-7.3){};
\node (14) at (-3.65,-7.6){{\small $a_{\tiny 4}$}};
\node[draw,fill,circle,inner sep=1pt] (5) at (-6.2,-7.3){};
\node (15) at (-6.15,-7.6){{\small $a_{\tiny 5}$}};
\node[draw,fill,circle,inner sep=1pt] (8) at (-5,-4){};
\node (18) at (-4.7,-3.95){{\small $b_{\tiny 1}$}};
\node[draw,fill,circle,inner sep=1pt] (14) at (-6,-4.8){};
\node (24) at (-6.3,-4.8){{\small $b_{\tiny 7}$}};
\node[draw,fill,circle,inner sep=1pt] (9) at (-4,-4.8){};
\node (19) at (-3.7,-4.8){{\small $b_{\tiny 2}$}};
\node[draw,fill,circle,inner sep=1pt] (15) at (-6.1,-5.8){};
\node (25) at (-6.30,-5.95){{\small $b_{\tiny 6}$}};
\node[draw,fill,circle,inner sep=1pt] (10) at (-3.9,-5.8){};
\node (20) at (-3.65,-5.95){{\small $b_{\tiny 3}$}};
\node[draw,fill,circle,inner sep=1pt] (12) at (-5.5,-6.6){};
\node (22) at (-5.8,-6.55){{\small $b_{\tiny 5}$}};
\node[draw,fill,circle,inner sep=1pt] (11) at (-4.5,-6.6){};
\node (21) at (-4.15,-6.6){{\small $b_{\tiny 4}$}};
\node (0) at (-5,-8.2){$\mathbf{P_{14}}$};

\draw[-] (1) -- (2)-- (3)--(4)--(5)--(6)--(7)--(14)--(9)--(11)--(15)--(8)--(10)--(12)--(14);
\draw[-](7)--(1);
\draw[-] (1) -- (8);
\draw[-] (9) -- (2);
\draw[-] (3) -- (10);
\draw[-] (4) -- (11);
\draw[-] (12) -- (5);
\draw[-] (15) -- (6);

%%%%%%%%%%%%%%%%%%%%%%%%%%%%%%%%%%%%%%%

\node (0) at (5,-8.6){$\mathbf{Q_{13}}$};
\node[draw,fill,circle,inner sep=1pt] (1) at (4,-3){};
\node (11) at (3.7,-3){{\small $a_{\tiny 1}$}};
\node[draw,fill,circle,inner sep=1pt] (2) at (6,-3){};
\node (12) at (6.3,-3){{\small $a_{\tiny 2}$}};
\node[draw,fill,circle,inner sep=1pt] (3) at (4,-4){};
\node (13) at (3.7,-4){{\small $d_{\tiny 1}$}};
\node[draw,fill,circle,inner sep=1pt] (5) at (6,-4){};
\node (15) at (6.3,-4){{\small $d_{\tiny 2}$}};
\node[draw,fill,circle,inner sep=1pt] (4) at (5,-4){};
\node (14) at (5.25,-4.25){{\small $h$}};
\node[draw,fill,circle,inner sep=1pt] (7) at (4,-6){};
\node (17) at (3.75,-6){{\small $g_{\tiny 1}$}};
\node[draw,fill,circle,inner sep=1pt] (9) at (6,-6){};
\node (19) at (6.3,-6){{\small $g_{\tiny 2}$}};
\node[draw,fill,circle,inner sep=1pt] (8) at (5,-6){};
\node (18) at (5.25,-5.75){{\small $h^{\prime}$}};
\node[draw,fill,circle,inner sep=1pt] (10) at (4,-7){};
\node (20) at (3.7,-7){{\small $b_{\tiny 1}$}};
\node[draw,fill,circle,inner sep=1pt] (11) at (6,-7){};
\node (21) at (6.3,-7){{\small $b_{\tiny 2}$}};
\node[draw,fill,circle,inner sep=1pt] (6) at (5,-5){};
\node (26) at (5.25,-5){{\small $v$}};
\node[draw,fill,circle,inner sep=1pt] (12) at (3,-8){};
\node (22) at (3,-8.3){{\small $c_{\tiny 1}$}};
\node[draw,fill,circle,inner sep=1pt] (13) at (7,-8){};
\node (23) at (7,-8.3){{\small $c_{\tiny 2}$}};

\node (24) at (5,-3.5){\textcolor{blue}{\small $C_{1}$}};
\node (25) at (5,-2.75){{\small $\mathbf{\tilde{e}_{\tiny 1}}$}};
\node (26) at (4.5,-7.75){{\small $\mathbf{\tilde{e}_{\tiny 2}}$}};
\node (27) at (5.5,-7.75){{\small $\mathbf{\tilde{e}_{\tiny 3}}$}};

\draw[-] (1) -- (2)-- (5)--(9)--(11)--(12)--(3)--(4)--(6)--(8)--(7)--(10)--(13)--(5)--(4);
\draw[-] (1) -- (3)-- (7);
\draw[-](8)--(9);
%%%%%%%%%%%%%%%%%%%%%%%%%%%%%%%%%%%%%%%%%%
\end{tikzpicture}
{\bf \caption{}\label{specialgraphs}}
\end{figure}

\begin{theorem}\label{wellcovered-characterization1}{\rm \cite[Theorem 1.1]{Finbow2}}
If $G$ is connected without $4$- and $5$-cycles, then $G$ is well-covered if and only if $G\in \{C_7,T_{10}\}$ or $\{V(H) \mid H\in S_G\}$ is a partition of $V(G)$.
\end{theorem}

\begin{remark}\rm\label{rem-nov13}
Suppose $G$ is well-covered. If $G$ is simplicial, or $G$ is chordal or $G$ is a graph without $4$- and $5$-cycles and $G\notin \{ C_7,T_{10}\}$. Then, by Theorems \ref{T1-T2Prisner} and \ref{wellcovered-characterization1}, $\{ V(H)\mid H \in S_G\}$ is a partition of $V(G)$. Therefore, $G$ is an $SCQ$ graph with $C_G=Q_G=\emptyset$. 
\end{remark}

\begin{theorem}\label{wellcovered-characterization2}{\rm \cite[Theorem 2 and Theorem 3]{Finbow}}
If $G$ is a connected graph without $3$- and $4$-cycles, then $G$ is well-covered if and only if $G\in\{K_1,C_7,P_{10},P_{13},P_{14},Q_{13}\}$ or $\{V(H) \mid H\in S_G \cup C_G\}$ is a partition of $V(G)$.
\end{theorem} 

\begin{definition}\rm
The {\it complement\/} of $G$, denoted by $\overline{G}$, is the graph with $V(\overline{G})=V(G)$ such that for each pair of different vertices $x$ and $y$ of $D$, we have that $\{ x,y\} \in E(\overline{G})$ if and only if $\{ x,y\} \notin E(G)$.
\end{definition}

\begin{definition}\rm\label{PerfectGraph}
A {\it $k$-colouring of $G$\/} is a function $c:V(G)\rightarrow \{ 1,2,\ldots ,k\}$ such that $c(u)\neq c(v)$ if $\{ u,v\} \notin E(G)$. The smallest integer $k$ such that $G$ has a $k$-colouring is called the {\it chromatic number of $G$\/} and it is denoted by $\chi (G)$. On the other hand, the {\it clique number\/}, denoted by $\omega(G)$ is the size of the largest complete subgraph of $G$. Finally, $G$ is {\it perfect\/} if $\chi (H)=\omega (H)$ for every induced subgraph $H$ of $G$.
\end{definition}

\begin{remark}\rm\label{stableset-char}
Let $A$ be a subset of $V(G)$, then $A$ is a stable set of $G$ if and only if $\overline{G}[A]$ is a complete subgraph of $\overline{G}$. Hence, $\beta(G)=\omega(\overline{G})$.
\end{remark}

\begin{theorem}\label{prop-Perfect}{\rm \cite[Theorem 5.5.3]{Diestel}}
$G$ is perfect if and only if $\overline{G}$ is perfect.
\end{theorem}

\section{Strong vertex cover and $\star$-semi-forest}
Let $D=(G,\mathcal{O},w)$ be a weighted oriented graph. In this Section, we introduce the unicycle oriented subgraphs (Definition \ref{Unicycle}), the root oriented trees (Definition \ref{ROT}), and the $\star$-semi-forests of $D$ (Definition \ref{semi-forest}). With this definitions, we characterize when a subset of $V(G)$ is contained in a strong vertex cover (see Theorem \ref{theorem-oct29}). Using this result, we characterize when $I(D)$ is unmixed if $G$ is a perfect graph (see Definition \ref{PerfectGraph} and Theorem \ref{Perf-Unm}). 

\begin{proposition}\label{GeneratinAStrongVC}
If $\mathcal{C}$ is a vertex cover of $D$ such that $N_{D}^{+}(A)\subseteq \mathcal{C}$ and $A\subseteq V^{+}$, then there is a strong vertex cover $\mathcal{C}^{\prime}$ of $D$, such that $N_{D}^{+}(A)\subseteq \mathcal{C}^{\prime}\subseteq \mathcal{C}$.
\end{proposition}
\begin{proof}
First, we prove that there is a vertex cover $\mathcal{C}^{\prime}$ such that $L_3(\mathcal{C}^{\prime})\subseteq N_{D}^{+}(A)\subseteq \mathcal{C}^{\prime}\subseteq \mathcal{C}$. We take $L:=N_{D}^{+}(A)$. If $L_{3}(\mathcal{C})\subseteq L$, then we take $\mathcal{C}^{\prime}=\mathcal{C}$. Now, we suppose there is $a_1\in L_{3}(\mathcal{C})\setminus L$, then by Remark \ref{VertexL3}, $N_{D}[a_1]\subseteq \mathcal{C}$. Thus, $\mathcal{C}_{1}=\mathcal{C}\setminus\{a_1\}$ is a vertex cover and $L\subseteq \mathcal{C}_1$, since $L\subseteq \mathcal{C}$ and $a_1\notin L$. Now, we suppose that there are vertex covers $\mathcal{C}_0,\ldots, \mathcal{C}_k$, such that $L\subseteq \mathcal{C}_i=\mathcal{C}_{i-1}\setminus \{a_i\}$ and $a_i\in L_{3}(\mathcal{C}_{i-1})\setminus L$ for $i=1,\ldots, k$ where $\mathcal{C}_0=\mathcal{C}$ and we give the following recursively process: If $L_{3}(\mathcal{C}_k)\subseteq L$, then we take $\mathcal{C}^{\prime}=\mathcal{C}_k$. Now, if there is $a_{k+1}\in L_{3}(\mathcal{C}_k)\setminus L$, then by Remark \ref{VertexL3}, $N_D[a_{k+1}]\subseteq \mathcal{C}_k$. Consequently, $\mathcal{C}_{k+1}:=\mathcal{C}_{k}\setminus \{a_{k+1}\}$ is a vertex cover. Also, $L\subseteq \mathcal{C}_{k+1}$, since $L\subseteq \mathcal{C}_k$ and $a_{k+1}\not\in L$. This process is finite, since $|V(D)|$ is finite. Hence, there is $m$ such that $L_3(\mathcal{C}_m)\subseteq L\subseteq \mathcal{C}_m\subseteq \mathcal{C}$. Therefore, we take $\mathcal{C}^{\prime}=\mathcal{C}_m$.  

\noindent
Now, we prove that $\mathcal{C}^{\prime}$ is strong. We take $x\in L_{3}(\mathcal{C}^{\prime})$, then $x\in L=N_{D}^{+}(A)$, since $L_3(\mathcal{C}^{\prime})\subseteq L$. Thus, $(y,x)\in E(D)$ for some $y\in A\subseteq V^{+}$. Hence, $y\in \mathcal{C}^{\prime}$, since $x\in L_{3}(\mathcal{C}^{\prime})$. Also, $y\not\in L_{1}(\mathcal{C}^{\prime})$, since $N_{D}^{+}(y)\subseteq N_{D}^{+}(A)\subseteq \mathcal{C}^{\prime}$. Hence, $y\in \big( \mathcal{C}^{\prime}\setminus L_1(\mathcal{C}^{\prime})\big) \cap V^{+}$. Therefore, $\mathcal{C}^{\prime}$ is strong. \qed
\end{proof}

\begin{definition}\rm\label{Unicycle}
If $B$ is a weighted oriented subgraph of $D$ with exactly one cycle $C$, then $B$ is called {\it unicycle oriented graph\/} when $B$ satisfies the following conditions:
\begin{enumerate}[noitemsep]
\item[$(i)$] $C$ is an oriented cycle in $B$ and there is an oriented path from $C$ to $y$ in $B$, for each $y\in V(B)\setminus V(C)$.
\item[$(ii)$] If $x\in V(B)$ with $w(x)=1$, then $deg_B(x)=1$. 
\end{enumerate}
\end{definition}

\begin{definition}\rm\label{ROT}
A weighted oriented subgraph $T$ of $D$ without cycles, is a {\it root oriented tree\/} ({\it ROT\/}) with {\it parent\/} $v\in V(T)$ when $T$ satisfies the following properties:
\begin{enumerate}[noitemsep]
\item[$(i)$] If $x\in V(T)\setminus \{ v \}$, there is an oriented path $\mathcal{P}$ in $T$ from $v$ to $x$.
\item[$(ii)$] If $x\in V(T)$ with $w(x)=1$, then $deg_T(x)=1$ and $x\neq v$ or $V(T)=\{ v\}$ and $x=v$. 
\end{enumerate} 
\end{definition} 

\begin{definition}\rm\label{semi-forest} 
A weighted oriented subgraph $H$ of $D$ is a {\it $\star$-semi-forest\/} if there are root oriented trees $T_1, \ldots , T_r$ whose parents are $v_1, \ldots , v_r$ and unicycle oriented subgraphs $B_1, \ldots , B_s$ such that $H=\big( \cup_{i=1}^{r} \ T_i \big) \cup \big( \cup _{j=1}^{s} \ B_j \big)$ with the following conditions:
\begin{enumerate}[noitemsep]
\item[$(i)$] $V(T_1), \ldots , V(T_r), V(B_1), \ldots , V(B_s)$ is a partition of $V(H)$. 
\item[$(ii)$] There is $W=\{ w_1, \ldots , w_r \} \subseteq V(D)\setminus V(H)$ such that $w_i \in N_{D}(v_i)$ for $i=1, \ldots , r$ (it is possible that $w_i=w_j$ for some $1\leqslant i<j\leqslant r$).  
\item[$(iii)$] There is a partition $W_1$, $W_2$ of $W$ such that $W_1$ is a stable set of $D$, $W_2\subseteq V^{+}$ and $(w_i , v_i)\in E(D)$ if $w_i \in W_2$. Also, $N_{D}^{+}( W_2\cup \tilde{H})\cap W_1=\emptyset$, where 
$$
\tilde{H}=\{ x\in V(H)\mid deg_H(x)\geqslant 2 \} \cup \{ v_i \mid deg_H(v_i)=1\}.
$$
\end{enumerate}
\end{definition}

\begin{remark}\rm\label{SubsetH}
By Definition \ref{Unicycle} and Definition \ref{ROT}, we have $\tilde{H} \subseteq V^{+}$. Furthermore, if $v_i$ is a parent vertex of $T_i$, with $deg_H(v_i)\geqslant 1$, then $v_i\in \tilde{H}$.
\end{remark}

\begin{lemma}\label{lemma-oct16}
If $H$ is a $\star$-semi-forest of $D$, then
\begin{center}
$V(H)\subseteq N_{D}(W_1)\cup N_{D}^{+}(W_2\cup \tilde{H})$.
\end{center}
\end{lemma}
\begin{proof}
We take $x\in V(H)$. Since $H=\big( \cup _{i=1}^{r} \ T_i \big) \cup \big( \cup _{j=1}^{s} \ B_j \big)$, we have two cases: \smallskip

\noindent
Case 1) $x\in V(B_j)$ for some $1\leqslant j\leqslant s$. Let $C$ be the oriented cycle of $B_j$. If $x\in V(C)$, then there is $y_1 \in V(C)$ such that $(y_1 , x)\in E(C)$. Furthermore, $deg_H (y_1)\geqslant deg_C (y_1)=2$, then $y_1 \in \tilde{H}$. Hence, $x\in N_{D}^{+}(y_1)\subseteq N_{D}^{+}(\tilde{H})$. Now, if $x\in V(B_j)\setminus V(C)$, then there is an oriented path $\mathcal{P}$ in $B_j$ from $C$ to $x$. Thus, there is $y_2 \in V(\mathcal{P})$ such that $(y_2, x)\in E(\mathcal{P})$. If $|V(\mathcal{P})|>2$, then $deg_H(y_2) \geqslant deg_{\mathcal{P}}(y_2)=2$. If $|V(\mathcal{P})|=2$, then $y_2 \in V(C)$ and $deg_H (y_2)>deg_C (y_2)=2$. Therefore, $y_2 \in \tilde{H}$ and $x\in N_{D}^{+}(\tilde{H})$. \smallskip

\noindent
Case 2) $x\in V(T_i)$ for some $1\leqslant i\leqslant r$. First, assume $x=v_i$, then there is $w_i \in W$ such that $x\in N_D(w_i)$. Consequently, $x\in N_D(W_1)$ if $w_i\in W_1$ and, by $(iii)$ of Definition \ref{semi-forest}, $x\in N_{D}^{+}(w_i)\subseteq N_{D}^{+}(W_2)$ if $w_i\in W_2$. Now, we suppose $x\neq v_i$, then there is an oriented path $\mathcal{L}$, from $v_i$ to $x$. Consequently, there is $y_3\in V(\mathcal{L})$ such that $(y_3, x)\in E(D)$. If $y_3\neq v_i$, then $deg_H(y_3)\geqslant deg_{\mathcal{L}}(y_3)=2$. Thus, $y_3\in \tilde{H}$ and $x\in N_{D}^{+}(\tilde{H})$. Finally, if $y_3=v_i$, then $deg_H(y_3)\geqslant 1$. Hence, by Remark  \ref{SubsetH}, $y_3\in \tilde{H}$ and $x\in N_{D}^{+}(\tilde{H})$. \qed 
\end{proof}

\begin{remark}\rm\label{Relat-W-H}
Sometimes to stress the relation between $W$ and $H$ in Definition \ref{semi-forest}, $W$ is denoted by $W^{H}$. Similarly, $W_1^{H}$ and $W_2^{H}$. If $\{ T_1,\ldots ,T_r\}=\emptyset$, then $W^{H}=W_1^{H}=W_2^{H}=\emptyset$.
\end{remark}

\begin{lemma}\label{lemma-dic3}
Let $K$ be a weighted oriented subgraph of $D$. If $H$ is a maximal ROT in $K$ with parent $v$, or $H$ is a maximal unicycle oriented subgraph in $K$ whose cycle is $C$, then there is no $(y,x)\in E(K)$ with $x\in V(K)\setminus V(H)$ and $y\in V^{+}\cap V(H)$.
\end{lemma}
\begin{proof}
By contradiction suppose there is $(y,x)\in E(K)$ with $x\in V(K)\setminus V(H)$ and $y\in V^{+}\cap V(H)$. Thus, $H\subsetneq H_1:=H\cup \{ (y,x)\} \subseteq K$. If $H$ is a unicycle oriented subgraph with cycle $C$ (resp. $H$ is a ROT), then there is an oriented path $\mathcal{P}$ from $C$ (resp. from $v$) to $y$. Consequently, $\mathcal{P}\cup \{ (y,x)\}$ is an oriented path from $C$ (resp. from $v$) to $x$ in $H_1$. Furthermore, $H_1$ has exactly one cycle (resp. has no cycles), since $deg_{H_1}(x)=1$ and $V(H_1)\setminus V(H)=\{ x\}$. 

\noindent
Now, we take $z\in V(H_1)$ with $w(z)=1$, then $z=x$ or $z\in V(H)$. We prove $deg_{H_1}(z)=1$. If $z=x$, then $deg_{H_1}(x)=1$. Now, if $z\in V(H)$, then $z\neq y$, since $y\in V^{+}$. So, $deg_{H_1}(z)=deg_H(z)$, since $N_{H_1}(x)=\{ y\}$. If $H$ is a ROT with $V(H)=\{ v\}$, then $y=z=v$. A contradiction, since $w(z)=1$ and $y\in V^{+}$. Consequently, by $(ii)$ in Definitions \ref{Unicycle} and \ref{ROT}, $deg_{H_1}(z)=deg_H(z)=1$. Hence, $H_1$ is a unicycle oriented subgraph with cycle $C$ (resp. is a ROT with parent $v$) of $K$. This is a contradiction, since $H\subsetneq H_1\subseteq K$ and $H$ is maximal. \qed
\end{proof}

\begin{definition}\rm
Let $K$ be a weighted oriented subgraph of $D$ and $H$ a $\star$-semi-forest of $D$. We say $H$ is a {\it generating $\star$-semi-forest\/} of $K$ if $V(K)=V(H)$.
\end{definition}

\begin{theorem}\label{theorem-oct29}
Let $K$ be an induced weighted oriented subgraph of $D$. Hence, the fo\-llo\-wing conditions are equivalent:
\begin{enumerate}[noitemsep]
\item[{\rm (1)}] There is a strong vertex cover $\mathcal{C}$ of $D$, such that $V(K) \subseteq\mathcal{C}$.
\item[{\rm (2)}] There is a generating $\star$-semi-forest $H$ of $K$. 
\end{enumerate}
\end{theorem}
\begin{proof}
${\rm (2)} \Rightarrow {\rm (1)}$
Let $\mathcal{C}_1$ be a minimal vertex cover of $D$. By {\rm (2)}, $K$ has a generating $\star$-semi-forest $H$. Now, using the notation of Definition \ref{semi-forest}, we take $\mathcal{C}_2=\big( \mathcal{C}_1\setminus W_1\big) \cup N_{D}(W_1)\cup N_{D}^{+}(W_2\cup \tilde{H})$. By Remark \ref{Compl-einC}, $\mathcal{C}_2$ is a vertex cover of $D$. Since $W_1$ is a stable set, $N_D(W_1)\cap W_1=\emptyset$. Then, $\mathcal{C}_2 \cap W_1=\emptyset$, since $N_{D}^{+}(W_2\cup \tilde{H})\cap W_1=\emptyset$. By Remark \ref{SubsetH} and $(iii)$ in Definition \ref{semi-forest}, $\tilde{H}\cup W_2\subseteq V^{+}$. So, by Proposition \ref{GeneratinAStrongVC}, there is a strong vertex cover $\mathcal{C}$ of $D$ such that $N_{D}^{+}(W_2\cup \tilde{H})\subseteq \mathcal{C} \subseteq \mathcal{C}_2$. Consequently, $\mathcal{C} \cap W_1=\emptyset$, since $\mathcal{C}_2 \cap W_1=\emptyset$. Thus, $N_D(W_1)\subseteq \mathcal{C}$, since $\mathcal{C}$ is a vertex cover. Then, by Lemma \ref{lemma-oct16}, $V(H)\subseteq N_D(W_1)\cup N_{D}^{+}(W_2\cup \tilde{H})\subseteq \mathcal{C}$. Hence, $V(K)\subseteq \mathcal{C}$, since $H$ is a generating $\star$-semi-forest  of $K$. \medskip

\noindent
${\rm (1)} \Rightarrow {\rm (2)}$
We have, $\mathcal{C}$ is a strong vertex cover such that $V(K)\subseteq \mathcal{C}$. If $A:=L_1(\mathcal{C})\cap V(K)=\{ v_1,\ldots , v_s\}$, then there is $w_i\in V(D)\setminus \mathcal{C}\subseteq V(D)\setminus V(K)$ such that $(v_i, w_i)\in E(D)$. We take the ROT's $M_1=\{ v_1\} ,\ldots ,M_s=\{ v_s\}$ and sets $W_{1}^{i}=\{ w_i\}$ and $W_{2}^{i}=\emptyset$ for $i=1, \ldots , s$. 

\noindent
Now, we will give a recursive process to obtain a generating $\star$-semi-forest of $K$. For this purpose, suppose we have connected $\star$-semi-forests $M_{s+1},\ldots ,M_l$ of $K\setminus A$ with subsets $W_{1}^{s+1},\ldots ,W_{1}^l, W_{2}^{s+1},\ldots ,W_{2}^l\subseteq V(D)\setminus V(K)$ and $V^{s+1},\ldots ,V^{l}\subseteq V(K)$ such that for each $s<j\leqslant l$, they satisfies the following conditions:
\begin{itemize}[noitemsep]
\item[{\rm (a)}] $V^{j}=\{ v_j\}$ if $M_j$ is a ROT with parent $v_j$ or $V^{j}$ is the cycle of $M_j$ if $M_j$ is a unicycle oriented subgraph,
\item[{\rm (b)}] $M_j$ is a maximal ROT in $K^{j}:=K\setminus \cup_{i=1}^{j-1} V(M_i)$ with parent in $V^{j}$ or $M_j$ is a maximal unicycle oriented subgraph in $K^{j}$ with cycle $V^{j}$.
\item[{\rm (c)}] $W_{1}^{j}\cap \mathcal{C}=\emptyset$ and $W_{2}^{j}\subseteq \big( \mathcal{C}\setminus (L_1(\mathcal{C})\cup V(K))\big)\cap V^{+}$.
\end{itemize}
Hence, we take $K^{l+1}:=K\setminus \big( \cup_{i=1}^{l} V(M_i)\big)$. This process starts with $l=s$; in this case, $K^{s+1}:=K\setminus \big( \cup_{i=1}^{s} V(M_i)\big) =K\setminus A$; furthermore, if $A=\emptyset$, then $K^{1}=K$. Continuing with the recursive process, if $K^{l+1}=\emptyset$, then $V(K)=\cup_{i=1}^{l} V(M_i)$ and we stop the process. Now, if $K^{l+1}\neq \emptyset$, then we will construct a connected $\star$-semi-forest $M_{l+1}$ of $K^{l+1}$ in the following way: \smallskip

\noindent
{\bf Case (1)} $L_2(\mathcal{C})\cap V(K^{l+1})\neq \emptyset$. Then, there is $z\in L_2(\mathcal{C})\cap V(K^{l+1})$. Thus, there is $(z^{\prime},z)\in E(D)$ with $z^{\prime}\notin \mathcal{C}$. We take a maximal ROT $M_{l+1}$ in $K^{l+1}$, whose parent is $z$. Also, we take $V^{l+1}=\{ v_{l+1}\} =\{ z\}$, $W_{1}^{l+1}=\{ w_{l+1}\} =\{ z^{\prime}\}$ and $W_{2}^{l+1}=\emptyset$. Hence, $M_{l+1}$ satisfies {\rm (a), (b)} and {\rm (c)}, since $z^{\prime}\notin \mathcal{C}$ and $W_{2}^{l+1}=\emptyset$. \smallskip

\noindent
{\bf Case (2)} $L_2(\mathcal{C})\cap V(K^{l+1})=\emptyset$. Then,	 $V(K^{l+1})\subseteq L_3(\mathcal{C})$, since $K^{l+1}\subseteq K\setminus A\subseteq \mathcal{C}\setminus L_1(\mathcal{C})$.  We take $x\in V(K^{l+1})$, then there is $x_1\in \big( \mathcal{C}\setminus L_1(\mathcal{C})\big) \cap V^{+}$ such that $(x_1,x)\in E(D)$, since $\mathcal{C}$ is strong. If $x_1\in V(K^{l+1})$, then there is $x_2\in \big( \mathcal{C}\setminus L_1(\mathcal{C})\big) \cap V^{+}$ such that $(x_2,x_1)\in E(D)$, since $\mathcal{C}$ is strong. Continuing with this process we obtain a maximal path $\mathcal{P}=(x_r,x_{r-1},\ldots ,x_1,x)$ such that $x_{r-1},\ldots ,x_1,x$ are different in $V(K^{l+1})$ and $x_1,\ldots ,x_r\in \big( \mathcal{C}\setminus L_1(\mathcal{C}) \big)\cap V^{+}$. Thus, $x_r\notin \cup_{j=1}^{s} V(M_j)$, since $x_r\notin L_1(\mathcal{C})$. Now, suppose $x_r\in V(M_j)$ for some $s<j\leqslant l$. So, $(x_r,x_{r-1})\in E(K^{j})$, $x_{r-1}\in V(K^{j})\setminus V(M_j)$ and $x_r\in V^{+}\cap V(M_j)$. Furthermore, $N_{D}^{+}(x_r)\cap W_{1}^{j}=\emptyset$, since $x_r\in \mathcal{C}\setminus L_1(\mathcal{C})$ and $\mathcal{C}\cap W_{1}^{j}=\emptyset$. A contradiction, by Lemma \ref{lemma-dic3}, since $W^{j}\cap V(K^{j})=\emptyset$. Hence, $x_r\notin \cup_{j=1}^{l} V(M_j)$. Consequently, $x_r\notin V(K)$ or $x_r\in V(K^{l+1})$. \smallskip

\noindent
\underline{Case (2.a)} $x_r\notin V(K)$. Then, take a maximal ROT $M_{l+1}$ in $K^{l+1}$ whose parent is $x_{r-1}$. Also, we take $V^{l+1}=\{ v_{l+1}\} =\{x_{r-1}\}$, $W_{1}^{l+1}=\emptyset$; and $W_{2}^{l+1}=\{ w_{l+1}\} =\{ x_r\}$. Thus, $M_{l+1}$ satisfies {\rm (a), (b)} and {\rm (c)}, since $W_{1}^{l+1}=\emptyset$, $x_r\in \big( \mathcal{C}\setminus L_1(\mathcal{C})\big) \cap V^{+}$ and $x_r\notin V(K)$. \smallskip

\noindent
\underline{Case (2.b)} $x_r\in V(K^{l+1})$. Then, $x_r\in L_3(\mathcal{C})$, since $V(K^{l+1})\subseteq L_3(\mathcal{C})$. Hence, there is $x_{r+1}\in \big( \mathcal{C}\setminus L_1(\mathcal{C})\big)\cap V^{+}$ such that $(x_{r+1},x_r)\in E(D)$, Then $\tilde{\mathcal{P}}=(x_{r+1},x_r,\ldots ,x_1,x)$ is an oriented walk. By the maximality of $\mathcal{P}$, we have that $x_r\in \{ x_{r-1},\ldots ,x_1,x\}$. Thus, $\mathcal{P}=(x_r,\ldots ,x_1,x)$ contains an oriented cycle $C$. We take a maximal unicycle oriented subgraph $M_{l+1}$ of $K^{l+1}$ with cycle $C$,
$V^{l+1}=C$ and $W_{1}^{l+1}=W_{2}^{l+1}=\emptyset$. Then, $M_{l+1}$ satisfies {\rm (a), (b)} and {\rm (c)}. \smallskip

\noindent
Since $K$ is finite, with this proceeding we obtain $M_1,\ldots , M_t\subseteq K$ such that $V(K)=\cup_{i=1}^{t} \ V(M_{i})$, $W_{1}^{i}\cap \mathcal{C}=\emptyset$ and $W_{2}^{i}\subseteq \big( \mathcal{C} \setminus L_1(\mathcal{C})\big) \cap V^{+}$ for $i=1,\ldots t$. We take $H:=\cup _{i=1}^{t} \ M_i$  with $W_j=\cup _{i=1}^{t} \ W_{j}^{i}$ for $j=1,2$. So, $V(H)=V(K)$. Also, $W_1 \cap \mathcal{C} =\emptyset$, then $W_1$ is a stable set, since $\mathcal{C}$ is a vertex cover. Furthermore, $W_2 \subseteq V^{+}$ and  $W_2\subseteq \mathcal{C} \setminus L_1(\mathcal{C})$, then $N_{D}^{+}(W_2)\subseteq \mathcal{C}$. Then, $N_{D}^{+}(W_2)\cap W_1=\emptyset$, since $\mathcal{C}\cap W_1=\emptyset$. If $x\in L_1(\mathcal{C})\cap V(K)$,  then there is $1\leqslant i\leqslant s$ such that $x=v_i$ and $M_i=\{ v_i\}$. Consequently, $deg_H(x)=deg_{M_i}(v_i)=0$. Thus, $\tilde{H} \cap L_1(\mathcal{C})=\emptyset$ implying $N_{D}^{+}(\tilde{H})\subseteq \mathcal{C}$, since $V(H)\subseteq \mathcal{C}$. Hence, $N_{D}^{+}(\tilde{H})\cap W_1=\emptyset$, since $W_1\cap \mathcal{C}=\emptyset$. Therefore, $H$ is a generating $\star$-semi-forest of $K$. \qed
\end{proof} 

\begin{theorem}\label{Perf-Unm}
Let $D=(G,\mathcal{O},w)$ be a weighted oriented graph where $G$ is a perfect graph, then $G$ has a $\tau$-reduction $H_1,\ldots ,H_s$ in complete subgraphs. Furthermore, $I(D)$ is unmixed if and only if each $H_i$ has no generating $\star$-semi-forests.
\end{theorem}
\begin{proof}
First, we prove $G$ has a $\tau$-reduction in complete graphs. By Theorem \ref{prop-Perfect}, $\overline{G}$ is perfect. Thus, $s:=w(\overline{G})=\chi(\overline{G})$. So, there is a $s$-colouring $c:V(\overline{G})\rightarrow \{ 1,\ldots ,s\}$. We take $V_i:=c^{-1}(i)$ for $i=1,\ldots ,s$. Then, $V_i$ is a stable set in $\overline{G}$, since $c$ is a $s$-colouring. Hence, by Remark \ref{stableset-char}, $H_i:=G[V_i]$ is a complete graph in $G$ and $s=\omega(\overline{G})=\beta(G)$. Furthermore, $V_1,\ldots ,V_s$ is a partition of $V(\overline{G})=V(G)$, since $c$ is a function. Consequently,
\begin{center}
$\sum\limits_{i=1}^{s} \tau (H_i)=\sum\limits_{i=1}^{s} \big( |V_i|-1\big) =\Big( \sum\limits_{i=1}^{s} |V_i|\Big) -s=|V(G)|-\beta (G)=\tau (G)$.
\end{center}
Finally, by Remark \ref{tau-beta}, $|V(G)|-\beta(G)=\tau(G)$, then,  $H_1,\ldots ,H_s$ is a $\tau $-reduction \linebreak of $G$. \medskip  

\noindent
Now, we prove that $I(D)$ is unmixed if and only if each $H_i$ has no generating $\star$-semi-forests.

\noindent
$\Rightarrow )$ 
By contradiction, assume $H_j$ has a generating $\star$-semi-forest, then by Theorem \ref{theorem-oct29} there is a strong vertex $\mathcal{C}$ such that $V_j\subseteq \mathcal{C}$. Furthermore, $\mathcal{C}\cap V_i$ is a vertex cover of $H_i$, then $|\mathcal{C}\cap V_i|\geqslant \tau (H_i)=|V_i|-1$ for $i\neq j$. Thus, $|\mathcal{C}|=\sum_{i=1}^{s}|\mathcal{C}\cap V_i|\geqslant |V_j|+\sum_{\substack{i=1\\ i\neq j}}^{s} (|V_i|-1)$, since $V_1,\ldots ,V_s$ is a partition of $V(G)$. Hence, by Remark \ref{tau-beta}, $|\mathcal{C}|>|V(G)|-s=\tau (G)$, since $s=\beta(G)$. A contradiction, by Remark \ref{MinimalStrongProp}, since $I(D)$ is unmixed. \medskip

\noindent
$\Leftarrow )$ Let $\mathcal{C}$ be a strong vertex cover, then $\mathcal{C}\cap V_i$ is a vertex cover of $H_i$. So, $|\mathcal{C} \cap V_i|\geqslant \tau (H_i)=|V_i|-1$ for $i=1, \ldots ,s$. Furthermore, by Theorem \ref{theorem-oct29}, $V_i\not\subseteq \mathcal{C}$. Consequently, $|\mathcal{C} \cap V_i|=|V_i|-1$. Thus, $|\mathcal{C}|=\sum_{i=1}^{s} \big( |V_i|-1\big)$, since $V_1,\ldots ,V_s$ is a partition of $V(G)$. Therefore, by {\rm (2)} in Theorem \ref{theorem42}, $I(D)$ is unmixed. \qed
\end{proof}

\section{Unmixedness of weighted oriented $SCQ$ graphs}
Let $D=(G,\mathcal{O},w)$ be a weighted oriented graph. If $P$ is a perfect matching of $G$ with the property {\bf (P)}, then in Proposition \ref{prop-unmixed}, we characterize when $|\mathcal{C}\cap e|=1$, for each strong vertex cover $\mathcal{C}$ of $D$ and each $e\in P$. Using Proposition \ref{prop-unmixed} in Corollary \ref{Koning-Unm}, we characterize when $I(D)$ is unmixed if $G$ is K\"oning. In Proposition \ref{Basic5Cycle-Equ}, we characterize the basic $5$-cycles, $C$ such that $|\mathcal{C}\cap V(C)|=3$ for each strong vertex cover $\mathcal{C}$ of $D$. Furthermore, in Theorem \ref{SCQ-char}, we characterize when $I(D)$ is unmixed if $G$ is an $SCQ$ graph (see Definition \ref{CondSCQ}). Finally, using this result we characterize the unmixed property of $I(D)$, when $G$ is simplicial or $G$ is chordal (see Corollary \ref{Simp-Chor-Unm}).

\begin{proposition}\label{prop-unmixed} 
Let $e$ be an edge of $G$. Hence, the following conditions are equivalent: 
\begin{enumerate}[noitemsep]
\item[{\rm (1)}] $|\mathcal{C}\cap e|=1$ for each strong vertex cover $\mathcal{C}$ of $D$.  
\item[{\rm (2)}] $e$ has the property {\bf (P)} and $N_D(b)\subseteq N_{D}^{+}(a)$ if $(a,b^{\prime})\in E(D)$ with $a\in V^{+}$ and $e=\{ b,b^{\prime} \}$.
 
\end{enumerate}
\end{proposition}
\begin{proof} 
${\rm (1)} \Rightarrow {\rm (2)}$
First, we show $e$ has the property {\bf (P)}. By contradiction, suppose there are $\{ a,b\}, \{ a^{\prime},b^{\prime} \} \in E(G)$ such that $\{ a,a^{\prime}\} \notin E(G)$. This implies, there is a maximal stable set $S$ such that $\{ a,a^{\prime}\} \subseteq S$. So, $\tilde{\mathcal{C}}=V(G)\setminus S$ is a minimal vertex cover. Consequently, $\tilde{\mathcal{C}}$ is strong. Furthermore, $a,a^{\prime} \notin \tilde{\mathcal{C}}$, then $b,b^{\prime}\in \tilde{\mathcal{C}}$, since $\{ a,b\},\{ a^{\prime},b^{\prime}\} \in E(G)$. A contradiction by {\rm (1)}.  Now, assume $(a,b^{\prime})\in E(D)$ with $a\in V^{+}$ and $e=\{ b,b^{\prime}\}$, then we will prove that $N_D(b)\subseteq N_{D}^{+}(a)$. By contradiction, suppose there is $c\in N_D(b)\setminus N_{D}^{+}(a)$. We take a maximal stable set $S$ such that $b\in S$. Thus, $\mathcal{C}_1=V(G)\setminus S$ is a minimal vertex cover such that $b\notin \mathcal{C}_1$. By Remark \ref{Compl-einC}, $\mathcal{C}=\big( \mathcal{C}_1\setminus \{ c \} \big)\cup N_{D}(c)\cup N_{D}^{+}(a)$ is a vertex cover. Furthermore, $c\notin \mathcal{C}$, since $c\notin N_{D}^{+}(a)$. By Proposition \ref{GeneratinAStrongVC}, there is a strong vertex cover $\mathcal{C}^{\prime}$ such that $N_{D}^{+}(a)\subseteq \mathcal{C}^{\prime}\subseteq \mathcal{C}$, since $a\in V^{+}$. Also, $b^{\prime}\in N_{D}^{+}(a)\subseteq \mathcal{C}^{\prime}$ and $c\notin \mathcal{C}^{\prime}$, since $(a,b^{\prime})\in E(D)$ and $c\notin \mathcal{C}$. Then, $b\in N_{D}(c)\subseteq \mathcal{C}^{\prime}$. Hence, $\{ b,b^{\prime} \} \subseteq \mathcal{C}^{\prime}$. This is a contradiction, by {\rm (1)}. \medskip

\noindent
${\rm (2)} \Rightarrow {\rm (1)}$
By contradiction, assume there is a strong vertex cover $\mathcal{C}$ of $D$ such that $|\mathcal{C}\cap e|\neq 1$. So, $|\mathcal{C}\cap e|=2$, since $\mathcal{C}$ is a vertex cover. Hence, by Theorem \ref{theorem-oct29}, there is a generating $\star$-semi-forest $H$ of $e$. We set $e=\{ z,z^{\prime}\}$. First, assume $H$ is not connected. Then, using the Definition \ref{semi-forest}, we have $H=M_1\cup M_2$ where $M_1=\{ v_1\}$, $M_2=\{ v_2\}$ and $w_1,w_2\in W$ such that $w_i\in N_D(v_i)$ for $i=1,2$. Thus, $\{ z,z^{\prime}\} =\{ v_1,v_2\}$ and $\{ w_1,w_2\} \in E(G)$, since $e$ satisfies the property {\bf(P)}. This implies $|W_1\cap \{ w_1,w_2\}|\leqslant 1$, since $W_1$ is a stable set. Hence, we can suppose $w_2\in W_2$, then $w_2\in V^{+}$ and $(w_2,z^{\prime})\in E(D)$. Consequently, by {\rm (2)}, $w_1\in N_D(z)\subseteq N_{D}^{+}(w_2)$, then $(w_2,w_1)\in E(D)$. Furthermore, by $(iii)$ in Definition \ref{semi-forest}, $N_{D}^{+}(W_2)\cap W_1=\emptyset$, then $w_1\in W_2$. So, $w_1\in V^{+}$ and $(w_1,z)\in E(D)$. By {\rm (1)} with $a=w_1$, we have $(w_1,w_2)\in E(D)$. A contradiction, then $H$ is connected. Thus, $H$ is a ROT with $V(H)=\{ z,z^{\prime}\}$. We can suppose $v_1=z$ and $W^{H}=\{ w_1\}$, then $(z,z^{\prime})\in E(D)$, $w_1\in N_D(z)$ and $z=v_1\in \tilde{H}$, since $deg_H(v_1)=1$. If $w_1\in N_{D}^{+}(z)$, then $w_1\in W_1$, since $z=v_1$. A contradiction, since $N_{D}^{+}(\tilde{H})\cap W_1=\emptyset$. Then, $w_1\notin N_{D}^{+}(z)$. By Remark \ref{SubsetH}, $z=v_1\in \tilde{H}\subseteq V^{+}$. Therefore, by {\rm (1)} (taking $a=b=z$ and $b^{\prime}=z^{\prime}$), we have $N_D(z)\subseteq N_{D}^{+}(z)$, since $e=\{ z,z^{\prime}\}$ and $z^{\prime}\in N_{D}^{+}(z)$. A contradiction, since $w_1\in N_D(z)\setminus N_{D}^{+}(z)$.  \qed 
\end{proof}

\begin{corollary}\label{Koning-Unm}{\rm \cite[Theorem 3.4]{V-R-P}}
Let $D=(G,\mathcal{O},w)$ be a weighted oriented graph, where $G$ is K\"oning without isolated vertices. Hence, $I(D)$ is unmixed if and only if $D$ satisfies the following two conditions:
\begin{enumerate}[noitemsep]
\item[{\rm (a)}] G has a perfect matching $P$ with the property {\bf (P)}.
\item[{\rm (b)}] $N_D(b)\subseteq N_{D}^{+}(a)$, when $a\in V^{+}$, $\{ b,b^{\prime}\} \in P$ and $b^{\prime}\in N_{D}^{+}(a)$.
\end{enumerate}
\end{corollary}
\begin{proof}
$\Rightarrow )$
By Theorem \ref{theorem42}, $I(G)$ is unmixed. Thus, by Remark \ref{1star} and Theorem \ref{Koning-Char}, $G$ has a perfect matching $P$ with the property {\bf (P)}. Consequently, $\nu(G)=|P|$. Also, $\tau(G)=\nu(G)$, since $G$ is K\"oning. So, $\tau(G)=|P|$. Now, we take a strong vertex cover $\mathcal{C}$ of $D$ and $e\in P$. Then, $|\mathcal{C}\cap e|\geqslant 1$. Furthermore, by Remark \ref{MinimalStrongProp}, $|\mathcal{C}|=\tau(G)=|P|$. Hence, $|\mathcal{C}\cap e|=1$, since $\mathcal{C}=\cup_{\tilde{e}\in P} \ \mathcal{C}\cap \tilde{e}$. Therefore, by Proposition \ref{prop-unmixed}, $D$ satisfies {\rm (b)}. \medskip

\noindent
$\Leftarrow )$
We take a strong vertex cover $\mathcal{C}$ of $D$. By Proposition \ref{prop-unmixed}, $|\mathcal{C}\cap e|=1$ for each $e\in P$, since $D$ satisfies {\rm (a)} and {\rm (b)}. This implies $|\mathcal{C}|=|P|$, since $P$ is a perfect matching. Therefore, by $(2)$ in Theorem \ref{theorem42}, $I(D)$ is unmixed. \qed
\end{proof}

\begin{lemma}\label{lemma-sep11}
If there is a basic $5$-cycle $C=(z_1, z_2, z_3, z_4, z_5, z_1)$ with $(z_{1},z_{2})$, $(z_{2},z_{3})\in E(D)$, $z_2 \in V^{+}$ and $C$ satisfies one of the following conditions:
\begin{enumerate}[noitemsep]
\item[{\rm (a)}] $(z_{3},z_{4})\in E(D)$ with $z_{3} \in V^{+}$.
\item[{\rm (b)}] $(z_{1},z_{5})$, $(z_{5},z_{4})\in E(D)$ with $z_{5} \in V^{+}$. 
\end{enumerate}
then there is a strong vertex cover $\tilde{\mathcal{C}}$ such that $|\tilde{\mathcal{C}}\cap V(C)|=4$.
\end{lemma}
\begin{proof}
We take $\mathcal{C}=\big( \mathcal{C}_0\setminus V(C)\big) \cup N_D (z_1) \cup N_{D}^{+}(z_2,x)$ where $\mathcal{C}_0$ is a vertex cover and $x=z_3$ if $C$ satisfies {\rm (a)} or $x=z_5$ if $C$ satisfies {\rm (b)}. Thus, $x\in V^{+}$. Furthermore, $z_2,z_3,z_5\in N_D(z_1)\cup N_{D}^{+}(z_2)$ and $z_4\in N_{D}^{+}(z_3)$ if $C$ satisfies {\rm (a)} or $z_4\in N_{D}^{+}(z_5)$ if $C$ satisfies {\rm (b)}. Hence, $\{ z_{2}, z_{3}, z_{4}, z_{5} \} \subseteq N_D (z_1)\cup N_{D}^{+}(z_2,x)$. Consequently, $\{ z_2,z_3,z_4,z_5\} \subseteq \mathcal{C}$, implying $\mathcal{C}$ is a vertex cover, since $\mathcal{C}_0$ is vertex cover and $N_D(z_1)\subseteq \mathcal{C}$. Also, $z_1 \notin \mathcal{C}$, since $z_1 \notin N_D (z_1)\cup N_{D}^{+}(z_2 , z_3)$ and $z_1\notin N_{D}^{+}(z_5)$ if $C$ satisfies {\rm (b)}. By Proposition \ref{GeneratinAStrongVC}, there is a strong vertex cover $\mathcal{C}^{\prime}$ such that $N_{D}^{+}(z_2,x)\subseteq \mathcal{C}^{\prime}\subseteq \mathcal{C}$, since $\{ z_2,x\} \subseteq V^{+}$. So, $z_1\notin \mathcal{C}^{\prime}$, since $z_1\notin \mathcal{C}$. Then, by Remark \ref{einC}, $N_D(z_1)\subseteq \mathcal{C}^{\prime}$. Hence, $\{ z_2,z_3,z_4,z_5\} \subseteq N_D(z_1)\cup N_{D}^{+}(z_2,x)\subseteq \mathcal{C}^{\prime}$. Therefore, $|\mathcal{C}^{\prime} \cap V(C)|=4$, since $z_1\notin \mathcal{C}^{\prime}$.   \qed
\end{proof}

\begin{definition}\rm
Let $C$ be an induced $5$-cycle, we say that $C$  has the {\it $\star$-property\/} if for each $(a,b)\in E(C)$ where $a\in V^{+}$, then $C=(a^{\prime},a,b,b^{\prime},c,a^{\prime})$ with the following properties:
\begin{enumerate}[noitemsep]
\item[$(\star .1)$] $(a^{\prime},a)\in E(D)$ and $w(a^{\prime})=1$.
\item[$(\star .2)$] $N_{D}^{-}(a)\subseteq N_{D}(c)$ and $N_{D}^{-}(a)\cap V^{+}\subseteq N_{D}^{-}(c)$.
\item[$(\star .3)$] $N_D(b^{\prime})\subseteq N_D(a^{\prime}) \cup N_{D}^{+}(a)$ and $N_{D}^{-}(b^{\prime})\cap V^{+}\subseteq N_{D}^{-}(a^{\prime})$.
\end{enumerate}
\end{definition}

\begin{lemma}\rm\label{NotStarCover}
Let  $C=(a_{1}^{\prime},a_1,b_1,b_{1}^{\prime},c_1,a_{1}^{\prime})$ be a basic $5$-cycle of $D$, such that $(a_{1}^{\prime},a_1)\in E(D)$, $deg_D(a_1)\geqslant 3$, $deg_D(c_1)\geqslant 3$ and $w(b_1)=1$. If there is a strong vertex cover $\mathcal{C}$ of $D$, such that $V(C)\subseteq \mathcal{C}$, then $C$ has no the $\star$-property.
\end{lemma}
\begin{proof}
By contradiction, suppose $C$ has the $\star$-property and there is a strong vertex cover $\mathcal{C}$, such that $V(C)\subseteq \mathcal{C}$. Then, $deg_D(a_{1}^{\prime})=deg_D(b_{1}^{\prime})=2$, since $C$ is a basic cycle, $deg_D(a_1)\geqslant 3$ and $deg_D(c_1)\geqslant 3$. Hence, $a_{1}^{\prime}, b_{1}^{\prime} \in L_3(\mathcal{C})$, since $V(C)\subseteq \mathcal{C}$. Thus, $(c_1,a_{1}^{\prime})\in E(D)$ and $w(c_1)\neq 1$, since $a_{1}^{\prime}\in L_3(\mathcal{C})$, $deg_D(a_{1}^{\prime})=2$, $(a_{1}^{\prime},a_1)\in E(D)$ and $\mathcal{C}$ is strong.  By ($\star .1$) with $(a,b)=(c_1,a_{1}^{\prime})$, we have that $(b_{1}^{\prime},c_1)\in E(D)$. Hence, $N_{D}^{-}(b_{1}^{\prime})\subseteq \{ b_1\}$, since $deg_D(b_{1}^{\prime})=2$. This is a contradiction, since $b_{1}^{\prime}\in L_3(\mathcal{C})$ and $w(b_1)=1$.  \qed
\end{proof}

\begin{proposition}\label{Basic5Cycle-Equ}
Let $C$ be a basic $5$-cycle, then $C$ has the $\star$-property if and only if $|\mathcal{C} \cap V(C)|=3$ for each strong vertex cover $\mathcal{C}$ of $D$.
\end{proposition}
\begin{proof}
\noindent
$\Rightarrow )$
By contradiction, we suppose there is a strong vertex cover $\mathcal{C}$ such that $|\mathcal{C}\cap V(C)|\geqslant 4$. Thus, there is a path $L=(d_1,d_2,d_3,d_4)\subseteq C$ such that $V(L)\subseteq \mathcal{C}$. Then, $deg_{D}(d_2)=2$ or $deg_{D}(d_3)=2$, since $C$ is basic. We can suppose $deg_D (d_2)=2$, then $N_D (d_2)\subseteq \mathcal{C}$. This implies $b_1 :=d_2 \in L_3 (\mathcal{C})$. So, there is $(a_1,b_1)\in E(D)$ with $a_1\in \big( \mathcal{C} \setminus L_{1}(\mathcal{C})\big) \cap V^{+}$, since $\mathcal{C}$ is strong. Since, $N_D(b_1)\subseteq C$, we can set $C=(a_{1}^{\prime},a_1,b_1,b_{1}^{\prime},c_1,a_{1}^{\prime})$. Consequently, $\{ a_1,b_{1}^{\prime}\}=N_D(b_1)=N_D(d_2)=\{ d_1,d_3 \} \subseteq \mathcal{C}$. By ($\star .1$), $(a_{1}^{\prime},a_1)\in E(D)$ and $w(a_{1}^{\prime})=1$. If $b_1\in V^{+}$, then by Remark \ref{rem-V-R-P}, $b_1$ is not a sink. This implies, $(b_1,b_{1}^{\prime})\in E(D)$. Then, by ($\star .1$) with $(a,b)=(b_1,b_{1}^{\prime})$, $w(a_1)=1$. A contradiction, since $a_1\in V^{+}$. Hence, $w(b_1)=1$.

\noindent 
We prove $a_{1}^{\prime}\in \mathcal{C}$. By contradiction assume $a_{1}^{\prime}\not\in \mathcal{C}$, then $\{b_1,a_1,c_1,b_{1}^{\prime}\} \subseteq \mathcal{C}$, since $|\mathcal{C}\cap V(C)|\geqslant 4$. Suppose $b_{1}^{\prime} \in L_3(\mathcal{C})$, then there is $y\in \big( N_{D}^{-}(b_{1}^{\prime})\cap V^{+} \big) \setminus L_1(\mathcal{C})$. Then, by ($\star .3$) with $(a,b)=(a_1,b_1)$, $y\in N_{D}^{-}(a_{1}^{\prime})$, i.e. $(y,a_{1}^{\prime})\in E(D)$. Consequently, $y\in L_1(\mathcal{C})$, since $a_{1}^{\prime}\notin \mathcal{C}$. This is a contradiction. Hence, $b_{1}^{\prime} \notin L_3(\mathcal{C})$, i.e. there is $y^{\prime}\in N_D(b_{1}^{\prime})\setminus \mathcal{C}$, since $b_{1}^{\prime} \in \mathcal{C}$. By ($\star .3$), $y^{\prime} \in N_D(a_{1}^{\prime})\cup N_{D}^{+}(a_1)$. Furthermore, $a_{1}^{\prime} \notin \mathcal{C}$, then $N_D(a_{1}^{\prime})\subseteq \mathcal{C}$ and $y^{\prime} \notin N_D(a_{1}^{\prime})$, since $\mathcal{C}$ is a vertex cover and $y^{\prime}\notin \mathcal{C}$. This implies $y^{\prime} \in N_{D}^{+}(a_1)$, then $a_1 \in L_1(\mathcal{C})$, since $a_1\in \mathcal{C}$ and $y^{\prime} \notin \mathcal{C}$. A contradiction, since $a_1 \notin L_1(\mathcal{C})$. Therefore, $a_{1}^{\prime}\in \mathcal{C}$.

\noindent
Thus, $\{b_1,a_1,a_{1}^{\prime},b_{1}^{\prime}\} \subseteq \mathcal{C}$. Now, we prove $c_1\in \mathcal{C}$, $deg_D(a_1)\geqslant 3$ and $deg_D(c_1)\geqslant 3$.

\noindent
{\bf Case (1)} $a_1\in L_3(\mathcal{C})$. Consequently, there is $z\in N_{D}^{-}(a_1)\cap V^{+}$ such that $z\in \mathcal{C}  \setminus L_1(\mathcal{C})$. Then, $z\notin V(C)$, since $N_{D}^{-}(a_1)\cap V(C)=\{ a_{1}^{\prime} \}$ and $w(a_{1}^{\prime})=1$. By ($\star .2$), $z\in N_{D}^{-}(c_1)$. Thus, $(z,c_1)\in E(D)$. Consequently, $c_1\in \mathcal{C}$, $deg_{D}(a_1)\geqslant 3$ and $deg_{D}(c_1)\geqslant 3$, since $z\in \mathcal{C} \setminus L_1(\mathcal{C})$ and $z\in N_D(a_1)\cap N_D(c_1)$. 

\noindent
{\bf Case (2)} $a_1\notin L_3(\mathcal{C})$. This implies, there is  $z^{\prime}\in N_D(a_1)$ such that $z^{\prime}\notin \mathcal{C}$.  Then, $z^{\prime}\notin V(C)$, since $N_D(a_1)\cap V(C)=\{ a_{1}^{\prime}, b_1\} \subseteq \mathcal{C}$. Consequently, $z^{\prime}\in N_{D}^{-}(a_1)$, since $a_1\in \mathcal{C}\setminus L_1(\mathcal{C})$. By ($\star .2$), we have $z^{\prime}\in N_{D}^{-}(a_1)\subseteq N_D(c_1)$. Hence, $c_1\in \mathcal{C}$, $deg_D(a_1)\geqslant 3$ and $deg_D(c_1)\geqslant 3$, since $z^{\prime}\notin \mathcal{C}$ and $z^{\prime}\in N_D(a_1)\cap N_D(c_1)$. 

\noindent
This implies, $V(C)\subseteq \mathcal{C}$. A contradiction, by Lemma \ref{NotStarCover}, since $C$ has the $\star$-property. \medskip

\noindent
$\Leftarrow )$
Assume $C=(a^{\prime},a,b,b^{\prime},c,a^{\prime})$ with $(a,b)\in E(C)$ such that $w(a)\neq 1$. We take a minimal vertex cover $\mathcal{C}$ of $D$. We will prove ($\star .1$), ($\star .2$) and ($\star .3$). \smallskip 

\noindent
$\mathbf{(\star .1)}$ First we will prove $(a^{\prime},a)\in E(D)$. By contradiction, suppose $(a,a^{\prime})\in E(D)$. By Remark \ref{rem-V-R-P}, there is $y\in N_{D}^{-}(a)$, since $a\in V^{+}$. Thus, $y\notin V(C)$ and $deg_{D}(a)\geq 3$. Consequently, $deg_{D}(a^{\prime})=deg_{D}(b)=2$, since $C$ is basic. Also, $deg_{D}(b^{\prime})=2$ or $deg_{D}(c)=2$, since $C$ is basic. We can assume $deg_{D}(c)=2$, then $N_D(c)=\{ a^{\prime},b^{\prime}\}$. So, by Remark \ref{Compl-einC}, $\mathcal{C}_1=\big(\mathcal{C}\setminus \{y,c\}\big) \cup N_{D}(y,b)\cup N_{D}^{+}(a)$ is a vertex cover, since $\mathcal{C}$ is a vertex cover, $\{ a^{\prime},b^{\prime}\} \subseteq N_D(b)\cup N_{D}^{+}(a)\subseteq \mathcal{C}_1$. Since $deg_{D}(c)=2$, we have $c\not\in N_{D}^{}(y)$. Furthermore, $c\notin N_D(b)\cup N_{D}^{+}(a)$, since $C$ is induced. Then, $c\not\in \mathcal{C}_1$. Also, $N_D(b)=\{ b^{\prime},a\}$, implies $y\not\in \mathcal{C}_1$, since $y\not\in N_{D}^{+}(a)$. By Proposition \ref{GeneratinAStrongVC} there is a strong vertex cover $\mathcal{C}_{1}^{\prime}$ such that $N_{D}^{+}(a)\subseteq \mathcal{C}_{1}^{\prime}\subseteq \mathcal{C}_{1}$, since $a\in V^{+}$. Thus, $c,y\notin \mathcal{C}_{1}^{\prime}$, since $c,y\notin\mathcal{C}_1$. By Remark \ref{einC}, $a^{\prime},b^{\prime},a\in N_D (c)\cup N_D (y)\subseteq \mathcal{C}_{1}^{\prime}$. Furthermore, $b\in N_{D}^{+}(a)\subseteq \mathcal{C}_{1}^{\prime}$. Hence, $|\mathcal{C}_{1}^{\prime}\cap V(C)|=4$. A contradiction.

\noindent
Now, we prove $w(a^{\prime})=1$. By contradiction, assume $w(a^{\prime})\neq 1$. By the last argument, $(c,a^{\prime})\in E(D)$, since $(a^{\prime},a)\in E(D)$ and $a\in V^{+}$. A contradiction, by {\rm (a)} in Lemma \ref{lemma-sep11}. \smallskip 

\noindent
$\mathbf{(\star .2)}$ We will prove $N_{D}^{-}(a)\subseteq N_{D}(c)$. By contradiction, suppose there is $y\in N_{D}^{-}(a)\setminus N_{D}(c)$. Also, $N_{D}^{-}(a)\cap V(C)\subseteq \{a^{\prime}\} \subseteq N_{D}(c)$, since $b\in N_{D}^{+}(a)$. Hence, $y\notin V(C)$. By Remark \ref{Compl-einC}, $\mathcal{C}_2=\big( \mathcal{C} \setminus \{y,c\}\big) \cup N_{D}(y,c)\cup N_{D}^{+}(a)$ is a vertex cover. Furthermore, $y,c\notin \mathcal{C}_2$, since $y\in N_{D}^{-}(a)\setminus N_D(c)$ and $c\notin N_D(a,y)$. By Proposition \ref{GeneratinAStrongVC}, there is a strong vertex cover $\mathcal{C}_{2}^{\prime}$ such that $N_{D}^{+}(a)\subseteq \mathcal{C}_{2}^{\prime} \subseteq \mathcal{C}_2$, since $a\in V^{+}$. Thus, $y,c\notin \mathcal{C}_{2}^{\prime}$ since $y,c\notin \mathcal{C}_2$. By Remark \ref{einC}, $a,a^{\prime},b^{\prime}\in N_D(y,c)\subseteq \mathcal{C}_{2}^{\prime}$. Hence, $|\mathcal{C}_{2}^{\prime}\cap V(C)|=4$, since $b\in N_{D}^{+}(a)\subseteq \mathcal{C}_{2}^{\prime}$. A contradiction. 

\noindent
Now, we prove $N_{D}^{-}(a) \cap V^{+} \subseteq N_{D}^{-}(c)$. By contradiction, suppose there is $y\in N_{D}^{-}(a)\cap V^{+}\setminus N_{D}^{-}(c)$. By Remark \ref{Compl-einC}, $\mathcal{C}_3=(\mathcal{C}\setminus \{c\})\cup N_{D}(c)\cup N_{D}^{+}(a,y)$ is a vertex cover. Furthermore, $c\notin N_{D}^{+}(a,y)$, then $c\notin \mathcal{C}_3$. By Proposition \ref{GeneratinAStrongVC}, there is a strong vertex cover $\mathcal{C}_{3}^{\prime}$ such that $N_{D}^{+}(a,y)\subseteq \mathcal{C}_{3}^{\prime}\subseteq \mathcal{C}_3$ since $\{ a,y\} \subseteq V^{+}$. So, $c\notin \mathcal{C}_{3}^{\prime}$, since $c\notin \mathcal{C}_3$. Thus, by Remark \ref{einC} $a^{\prime},b^{\prime}\in N_D (c)\subseteq \mathcal{C}_{3}^{\prime}$. Also, $a,b\in N_{D}^{+}(a,y)\subseteq \mathcal{C}_{3}^{\prime}$. Hence, $|\mathcal{C}^{\prime}_{3}\cap V(C)|= 4$, a contradiction. \smallskip

\noindent
$\mathbf{(\star .3)}$ We prove $N_D(b^{\prime})\subseteq N_D(a^{\prime})\cup N_{D}^{+}(a)$. By contradiction, we suppose there is $y\in N_D(b^{\prime})\setminus \big( N_D(a^{\prime})\cup N_{D}^{+}(a)\big)$. Thus, $y\notin C$, since $N_D (b^{\prime})\cap V(C)=\{ c,b\} \subseteq N_D (a^{\prime}) \cup N_{D}^{+}(a)$. By Remark \ref{Compl-einC}, $\mathcal{C}_4 =\big( \mathcal{C}\setminus \{ y,a^{\prime}\})\big) \cup N_D(y,a^{\prime})\cup N_{D}^{+} (a)$. Furthermore, $y\notin \mathcal{C}_4$, since $y\notin N_D(a^{\prime})\cup N_{D}^{+}(a)$. By ($\star .1$), $(a^{\prime},a)\in E(D)$, then $a^{\prime} \notin \mathcal{C}_4$, since $a^{\prime} \notin N_D(y) \cup N_{D}^{+}(a)$. By Proposition \ref{GeneratinAStrongVC}, there is a strong vertex cover $\mathcal{C}_{4}^{\prime}$ such that $N_{D}^{+}(a)\subseteq \mathcal{C}_{4}^{\prime} \subseteq \mathcal{C}_4$, since $a\in V^{+}$. So, $y,a^{\prime} \notin \mathcal{C}_{4}^{\prime}$, since $y,a^{\prime} \notin \mathcal{C}_4$. Thus, by Remark \ref{einC} $b^{\prime},a,c\in N_D(y)\cup N_D(a^{\prime})\subseteq \mathcal{C}_{4}^{\prime}$. Also, $b\in N_{D}^{+} (a) \subseteq \mathcal{C}_{4}^{\prime}$. Hence, $|\mathcal{C}_{4}^{\prime} \cap V(C)|=4$, a contradiction. \smallskip

\noindent
Finally, we prove $N_{D}^{-}(b^{\prime})\cap V^{+} \subseteq N_{D}^{-}(a^{\prime})$.  By contradiction, we suppose there is $y\in \big( N_{D}^{-}(b^{\prime})\cap V^{+}\big) \setminus N_{D}^{-}(a^{\prime})$. By $(\star .1)$, $a^{\prime} \in N_{D}^{-}(a)$. Furthermore, by {\rm (a)} in Lemma \ref{lemma-sep11}, $y\neq b$, since $y\in V^{+}$. If $y=c$, then $(c,b^{\prime})\in E(D)$. Thus, by $(\star .1)$, with the edge $(a^{\prime},c)\in E(D)$. A contradiction by {\rm (b)} in Lemma \ref{lemma-sep11}, since $c=y\in V^{+}$. Hence, $y\notin V(C)$. By Remark \ref{Compl-einC}, $\mathcal{C}_5=\big( \mathcal{C}\setminus \{ a^{\prime}\}\big) \cup N_{D}(a^{\prime})\cup N_{D}^{+}(y,a)$ is a vertex cover. By Remark \ref{Compl-einC}, $a^{\prime} \notin N_{D}^{+}(y,a)$, since $(a^{\prime},a)\in E(D)$ and $y\notin N_{D}^{-}(a^{\prime})$. Consequently, $a^{\prime} \notin \mathcal{C}_5$. By Proposition \ref{GeneratinAStrongVC}, there is a strong vertex cover $\mathcal{C}_{5}^{\prime}$ such that $N_{D}^{+} (a,y)\subseteq \mathcal{C}_{5}^{\prime} \subseteq \mathcal{C}_5$, since $\{ a,y\} \subseteq V^{+}$. So, $a^{\prime} \notin \mathcal{C}_{5}^{\prime}$, since $a^{\prime} \notin \mathcal{C}_{5}$. Then, by Remark \ref{einC} $a,c\in N_{D}(a^{\prime}) \subseteq \mathcal{C}_{5}^{\prime}$. Furthermore, $b,b^{\prime} \in N_{D}^{+}(a,y) \subseteq \mathcal{C}_{5}^{\prime}$. Hence, $|\mathcal{C}_{5}^{\prime} \cap V(C)|=4$, a contradiction. \qed \smallskip
\end{proof}

\begin{lemma}\label{lemma-08ene}
Let $\mathcal{C}$ be a vertex cover of $D$ where $G$ is an $SCQ$ graph. Hence, $|\mathcal{C}|=\tau(G)$ if and only if $|\mathcal{C}\cap V(K)|=|V(K)|-1$, $|\mathcal{C}\cap V(C)|=3$ and $|\mathcal{C}\cap e|=1$ for each $K\in S_G$, $C\in C_G$ and $e\in Q_G$, respectively.
\end{lemma}
\begin{proof}
We set $\mathcal{C}$ a vertex cover of $D$, $K\in S_G$, $C\in C_G$ and $e\in Q_G$. Then, there are $y\in V(G)$ and $a,a^{\prime}\in V(C)$ such that $K=G[N_G[y]]$, $deg_G(a)=deg_G(a^{\prime})=2$ and $\{ a,a^{\prime}\} \notin E(G)$. We set $A_K:=V(K)\setminus \{ y\}$ and $B_C:=V(C)\setminus \{ a,a^{\prime}\}$. Also, $\mathcal{C}\cap V(K)$ is a vertex cover of $K$, so $|\mathcal{C}\cap V(K)|\geqslant\tau(K)=|V(K)|-1$. Similarly, $|\mathcal{C}\cap V(C)|\geqslant\tau(C)=3$ and $|\mathcal{C}\cap e|\geqslant\tau(e)=1$. Thus, 
\begin{equation}\label{tau-SCQ}
|\mathcal{C}|=\sum\limits_{K\in S_G} |\mathcal{C}\cap V(K)|+\sum\limits_{C\in C_G} |\mathcal{C}\cap V(C)|+\sum\limits_{e\in Q_G} |\mathcal{C}\cap e|\geqslant \sum\limits_{K\in S_G} \big( |V(K)|-1\big)+3|C_G|+|Q_G|,
\end{equation} 
since $\mathcal{H}=\{ V(H)\mid H\in S_G\cup C_G\cup Q_G\}$ is a partition of $V(G)$. Now, we take a maximal stable set $S$ contained in $V(Q_G):=\{ x\in V(G)\mid x\in e \ {\rm and} \ e\in Q_G\}$. Then, $|S\cap e|\leqslant 1$ for each $e\in Q_G$, since $S$ is stable. If $S\cap e=\emptyset$ for some $e=\{ x_1,x_2\} \in Q_G$, then there are $y_1,y_2\in S$ such that $\{ x_1,y_1\}$, $\{ x_2,y_2\} \in E(G)$, since $S$ is maximal. But $Q_G$ satisfies the property {\bf (P)}, then $\{ y_1,y_2\}\in E(G)$. A contradiction, since $S$ is stable. Hence, $|S\cap e|=1$ for each $e\in Q_G$. Consequently, $|S|=|Q_G|$ and $|S^{\prime}|=|Q_G|$, where $S^{\prime}=V(Q_G)\setminus S$. Now, we take 
\begin{center}
$\mathcal{C}(S^{\prime})= \Big( \bigcup\limits_{K\in S_G} A_K \Big) \bigcup \Big( \bigcup\limits_{C\in C_G} B_C \Big) \bigcup S^{\prime}$.
\end{center}
We prove $\mathcal{C}(S^{\prime})$ is a vertex cover of $D$. By contradiction, suppose there is $\hat{e}\in E(G)$ such that $\hat{e}\cap \mathcal{C}(S^{\prime})=\emptyset$. We set $z\in \hat{e}$, then $\hat{e}=\{ z,z^{\prime}\}$. If $z\in V(\tilde{K})$ for some $\tilde{K}\in S_G$, \linebreak then $\tilde{K}=G[N_G[z]]$, since $A_{\tilde{K}}\subseteq \mathcal{C}(S^{\prime})$ and $z\notin \mathcal{C}(S^{\prime})$. So, $z^{\prime}\in N_G(z)\subseteq 
\tilde{K}\setminus \{ z\}=A_{\tilde{K}}\subseteq \mathcal{C}(S^{\prime})$. A contradiction, since $\hat{e}\cap \mathcal{C}(S^{\prime})=\emptyset$. Now, if $z\in V(\tilde{C})$ for some $\tilde{C}\in C_G$, then $z\notin B_{\tilde{C}}$. Thus, $deg_G(z)=2$ implying $z^{\prime}\in B_{\tilde{C}}\subseteq \mathcal{C}(S^{\prime})$, since $\{ z,z^{\prime}\}\in E(G)$. A contradiction. Then, $\hat{e}\subseteq V(Q_G)$, since $\mathcal{H}$ is a partition of $V(G)$. Also, $\hat{e}\cap S^{\prime}=\emptyset$, this implies $\hat{e}\subseteq V(Q_G)\setminus S^{\prime}=S$. But $S$ is stable. This is a contradiction. Hence, $\mathcal{C}(S^{\prime})$ is a vertex cover of $D$. Furthermore,
\begin{center}
$|\mathcal{C}(S^{\prime})|=\sum\limits_{K\in S_G} |A_K|+\sum\limits_{C\in C_G} |B_C|+|S^{\prime}|=\sum\limits_{K\in S_G} \big( |V(K)|-1\big)+3|C_G|+|Q_G|$.
\end{center}
Thus, $\tau(G)=\sum_{K\in S_G} \big( |V(K)|-1\big)+3|C_G|+|Q_G|$. Therefore, by (\ref{tau-SCQ}), $|\mathcal{C}|=\tau(G)$ if and only if $|\mathcal{C}\cap V(K)|=|K|-1$, $|\mathcal{C}\cap V(C)|=3$ and $|\mathcal{C}\cap e|=1$ for each $K\in S_G$, $C\in C_G$ and $e\in Q_G$, respectively.
\qed
\end{proof}

\begin{theorem}\label{SCQ-char}
Let $D=(G,\mathcal{O},w)$ be a weighted oriented graph where $G$ is an $SCQ$ graph. Hence, $I(D)$ is unmixed if and only if $D$ satisfies the following conditions:
\begin{enumerate}[noitemsep]
\item[{\rm (a)}] Each basic $5$-cycle of $G$ has the $\star$-property.
\item[{\rm (b)}] Each simplex of $D$ has no generating $\star$-semi-forests.
\item[{\rm (c)}] $N_D(b)\subseteq N_{D}^{+}(a)$ when $a\in V^{+}$, $\{ b,b^{\prime} \} \in Q_G$ and $b^{\prime} \in N_{D}^{+}(a)$.
\end{enumerate}
\end{theorem}
\begin{proof}
$\Rightarrow )$
We take a strong vertex cover $\mathcal{C}$ of $D$, then by Remark \ref{MinimalStrongProp}, $|\mathcal{C}|=\tau (G)$. Consequently, by Lemma \ref{lemma-08ene}, $|\mathcal{C}\cap V(K)|=|V(K)|-1$, $|\mathcal{C}\cap V(C)|=3$ and $|\mathcal{C}\cap e|=1$ for each $K\in S_G$, $C\in C_G$ and $e\in Q_G$. Thus, $V(K)\not\subseteq \mathcal{C}$. Consequently, by Theorem \ref{theorem-oct29}, $D$ satisfies {\rm (b)}. Furthermore, by Propositions \ref{Basic5Cycle-Equ} and \ref{prop-unmixed}, $D$ satisfies {\rm (a)} and {\rm (c)}. \medskip

\noindent
$\Leftarrow )$
Let $\mathcal{C}$ be a strong vertex cover of $D$. By {\rm (a)} and Proposition \ref{Basic5Cycle-Equ}, we have $|\mathcal{C}\cap V(C)|=3$ for each $C\in C_G$. Furthermore, by {\rm (b)} and Theorem \ref{theorem-oct29}, $V(K)\not\subseteq \mathcal{C}$ for each $K\in S_G$. Consequently, $|V(K)|>|\mathcal{C} \cap V(K)|\geqslant \tau (K)=|V(K)|-1$. So, $|\mathcal{C}\cap V(K)|=|V(K)|-1$. Now, if $e\in Q_G$, then $e$ has the property {\bf (P)}, since $Q_G$ has the property {\bf (P)}. Thus, by {\rm (c)} and Proposition \ref{prop-unmixed}, $|\mathcal{C} \cap e|=1$. Hence, by Lemma \ref{lemma-08ene}, $|\mathcal{C}|=\tau (G)$. Therefore $I(D)$ is unmixed, by $(2)$ in Theorem \ref{theorem42}. \qed
\end{proof}

\begin{corollary}\label{Simp-Chor-Unm}
Let $D=(G,\mathcal{O},w)$ be a weighted oriented graph where $G$ is a simplicial or chordal graph. Hence, $I(D)$ is unmixed if and only if $D$ satisfies the following conditions:
\begin{enumerate}[noitemsep]
\item[{\rm (a)}] Each vertex is in exactly one simplex of $D$.
\item[{\rm (b)}] Each simplex of $D$ has not a generating $\star$-semi-forest.
\end{enumerate} 
\end{corollary}
\begin{proof}
$\Rightarrow )$ By $(3)$ in Theorem \ref{theorem42} and Remark \ref{1star}, $G$ is well-covered. Thus, by Theorem \ref{T1-T2Prisner}, $G$ sa\-tis\-fies {\rm (a)}. Furthermore, by Remark \ref{rem-nov13}, $G$ is an $SCQ$ graph with $C_G=Q_G=\emptyset$. Hence, by Theorem \ref{SCQ-char}, $D$ satifies {\rm (b)}. \medskip

\noindent
$\Leftarrow )$ By {\rm (a)}, $\{ V(H)\mid H\in S_G\}$ is a partition of $V(G)$. Hence, $G$ is an $SCQ$ graph with $C_G=\emptyset$ and $Q_{G}=\emptyset$. Therefore, by {\rm (b)} and Theorem \ref{SCQ-char}, $I(D)$ is unmixed. \qed
\end{proof}

\section{Unmixedness of weighted oriented graphs without some small cycles}
Let $D=(G,\mathcal{O},w)$ be a weighted oriented graph. In this Section, we study and cha\-rac\-te\-ri\-ze the unmixed property of $I(D)$ when $G$ has no $3$- or $5$- cycles (Theorem \ref{theorem-oct31}), or $G$ is a graph without $4$- or $5$-cycles (Theorem \ref{No4,5Cyc-Unmix}), or $girth(G)\geqslant 5$ (Theorem \ref{No3,4,5Cyc-Unmix}). In other words, in this Section, we characterize the unmixed property of $I(D)$ when $G$ has at most one of the following types of cycles: $3$-cycles, $4$-cycles and $5$-cycles.

\begin{proposition}\label{prop-oct31}
If for each $(y,x)\in E(D)$ with $y\in V^{+}$, we have that $N_D(y^{\prime})\subseteq N_{D}^{+}(y)$ for some $y^{\prime}\in N_D(x)\setminus y$, then $L_3(\mathcal{C})=\emptyset$ for each strong vertex cover $\mathcal{C}$ of $D$.
\end{proposition}
\begin{proof}
By contradiction, suppose there is  a strong vertex cover $\mathcal{C}$ of $D$ and $x\in L_3(\mathcal{C})$. Hence, there is $y\in \big( \mathcal{C}\setminus L_1(\mathcal{C})\big) \cap V^{+}$ with $(y,x)\in E(D)$. Then, $N_D(x)\subseteq \mathcal{C}$ and $N_{D}^{+}(y)\subseteq \mathcal{C}$, since $x\in L_3(\mathcal{C})$ and $y\in \mathcal{C}\setminus L_1(\mathcal{C})$. By hypothesis, there is a vertex $y^{\prime}\in N_D(x)\setminus y\subseteq \mathcal{C}$ such that $N_D(y^{\prime})\subseteq N_{D}^{+}(y)\subseteq \mathcal{C}$. Thus, $y^{\prime}\in L_3(\mathcal{C})$. Since $\mathcal{C}$ is strong, there is $(y_1,y^{\prime})\in E(D)$ with $y_1\in V^{+}$. So, $y_1\in N_D(y^{\prime})\subseteq N_{D}^{+}(y)$. On the other hand, $(y_1,x_1)\in E(D)$ where $x_1:=y^{\prime}$ and $y_1\in V^{+}$, then by hypothesis, there is $y_{1}^{\prime}\in N_D(x_1)\setminus y_1$ such that $N_D(y_{1}^{\prime})\subseteq N_{D}^{+}(y_1)$. Hence, $y_{1}^{\prime}\in N_D(x_1)=N_D(y^{\prime})\subseteq N_{D}^{+}(y)$. Consequently, $y\in N_D(y_{1}^{\prime})\subseteq N_{D}^{+}(y_1)$. A contradiction, since $y_1\in N_{D}^{+}(y)$. \qed
\end{proof}

\begin{corollary}\label{cor-oct30}
If $G$ is well-covered and $V^{+}$ is a subset of sinks, then $I(D)$ is unmixed.
\end{corollary}
\begin{proof}
If $y\in V^{+}$, then $y$ is a sink. Thus, $(y,x)\notin E(D)$ for each $x\in V(D)$. Hence, by Proposition \ref{prop-oct31}, $L_3(\mathcal{C})=\emptyset$, for each strong vertex cover $\mathcal{C}$ of $D$. Furthermore, by Remark \ref{1star}, $I(G)$ is unmixed. Therefore $I(D)$ is unmixed, by $(3)$ in Theorem \ref{theorem42}. \qed
\end{proof}

\begin{lemma}\label{StableSet}
Let $(z,y), (y,x)$ be edges of $D$ with $y\in V^{+}$ and $N_D(x)=\{ y,x_1,\ldots ,x_s\}$. If there are $z_i\in N_D(x_i)\setminus N_{D}^{+}(y)$ such that $\{ z,x,z_1,\ldots ,z_s\}$ is a stable set, then $I(D)$ is mixed.
\end{lemma}
\begin{proof}
We take $A:=\{ z,z_1,\ldots ,z_s\}$, then $A\cup \{ x\}$ is a stable set. We can take a maximal stable set $S$ of $V(G)$, such that $A\cup \{ x\}\subseteq S$. So, $\tilde{\mathcal{C}}=V(G)\setminus S$ is a minimal vertex cover of $D$. Hence, $\mathcal{C}=\tilde{\mathcal{C}} \cup N_{D}^{+}(y)$ is a vertex cover of $D$. Also $A\cap \mathcal{C}=\emptyset$, since $A\subseteq S$, $z\in N_{D}^{-}(y)$ and $z_i\notin N_{D}^{+}(y)$. By Proposition \ref{GeneratinAStrongVC}, there is a strong vertex cover $\mathcal{C}^{\prime}$ of $D$ such that $N_{D}^{+}(y)\subseteq \mathcal{C}^{\prime} \subseteq \mathcal{C}$, since $y\in V^{+}$. Thus, $A\cap \mathcal{C}^{\prime}=\emptyset$, since $A\cap \mathcal{C}=\emptyset$. Then, by Remark \ref{einC}, $N_D(A)\subseteq \mathcal{C}^{\prime}$. Furthermore $N_D(x)=\{ y,x_1,\ldots ,x_s\} \subseteq N_D(A)$. Consequently, $N_D(x)\subseteq \mathcal{C}^{\prime}$. Hence, $x\in L_3(\mathcal{C}^{\prime})$, since $x\in N_{D}^{+}(y)\subseteq \mathcal{C}^{\prime}$. Therefore, by {\rm (3)} in Theorem \ref{theorem42}, $I(D)$ is mixed. \qed
\end{proof}

\begin{theorem}\label{theorem-oct31}
Let $D=(G,\mathcal{O},w)$ be a weighted oriented graph such that $G$ has no $3$- or $5$-cycles. Hence, $I(D)$ is unmixed if and only if $D$ satisfies the following conditions:
\begin{enumerate}[noitemsep]
\item[{\rm (a)}] $G$ is well-covered.
\item[{\rm (b)}] If $(y,x)\in E(D)$ with $y\in V^{+}$, then $N_D(y^{\prime})\subseteq N_{D}^{+}(y)$ for some $y^{\prime}\in N_D(x)\setminus y$.
\end{enumerate}
\end{theorem}
\begin{proof}
$\Leftarrow )$ By Proposition \ref{prop-oct31} and {\rm (b)}, we have that $L_3(\mathcal{C})=\emptyset$ for each strong vertex cover $\mathcal{C}$ of $D$. Furthermore, by {\rm (a)} and Remark \ref{1star}, $I(G)$ is unmixed. Therefore, by $(3)$ in Theorem \ref{theorem42}, $I(D)$ is unmixed.\medskip

\noindent
$\Rightarrow )$ By {\rm (3)} in Theorem \ref{theorem42} and Remark \ref{1star}, $D$ satisfies {\rm (a)}. Now, we take $(y,x)\in E(D)$ with $y\in V^{+}$. Then, by Remark \ref{rem-V-R-P}, there is $z\in N_{D}^{-}(y)$. Furthermore $z \notin N_D(x)$, since $G$ has no $3$-cycles. We set $N_D(x)\setminus y=\{ x_1,\ldots ,x_s\}$. We will prove {\rm (b)}. By contradiction, suppose there is $z_i\in N_D(x_i)\setminus N_{D}^{+}(y)$ for each $i=1,\ldots , s$. If $\{ z_i,z_j\} \in E(G)$ for some $1\leqslant i<j\leqslant s$, then $(x,x_i,z_i,z_j,x_j,x)$ is a $5$-cycle. But $G$ has no $5$-cycles, then $\{ z_1,\ldots , z_s\}$ is a stable set. Now, if $\{ x,z_k\} \in E(G)$ or $\{ z,z_k\} \in E(G)$ for some $k\in \{1,\ldots ,s\}$, then $(x,x_k,z_k,x)$ is a $3$-cycle or $(z,y,x,x_k,z_k,z)$ is a $5$-cycle. Hence, $\{ x,z,z_1,\ldots , z_s\}$ is a stable set. A contradiction, by Lemma \ref{StableSet}, since $I(D)$ is unmixed. \qed
\end{proof}

\noindent
In the following results, we use the notation of Figure \ref{specialgraphs}.

\begin{remark}\rm\label{Not3-4-cycles}
Let $G$ be a graph in $\{ C_7,T_{10},P_{10},P_{13},P_{14},Q_{13}\}$. Hence,
\begin{enumerate}[noitemsep]
\item[{\rm (a)}] $G$ does not contain $4$-cycles. Furthermore, if $G$ has a $3$-cycle, then $G=T_{10}$. \smallskip
\item[{\rm (b)}] If $deg_G(x)=2$, then $x$ is not in a $3$-cycle of $G$. \smallskip
\item[{\rm (c)}] If $G\neq C_7$ and $\tilde{e}=\{ v,u\} \in E(G)$ with $deg_D(v)=deg_D(u)=2$, then $\tilde{e}\in \{ \tilde{e}_1,\tilde{e}_2,\tilde{e}_3 \}$. Also, if $\tilde{e}$ is in a $5$-cycle $C$, then $G\in \{ P_{10},P_{13}\}$, $\tilde{e}\in \{ \tilde{e}_1,\tilde{e}_2  \}$ and $C\in \{ C_1,C_2\}$ or $G=Q_{13}$, $\tilde{e}=\tilde{e}_1$ and $C=C_1$. \smallskip
\item[{\rm (d)}] If $P=(y_1,y_2,y_3)$ is a path in $G$ with $deg_G(y_i)=2$ for $i=1,2,3$, then $G=C_7$. 
\end{enumerate}
\end{remark}
\begin{proof} 
{\rm (a)} By Theorems \ref{wellcovered-characterization1} and \ref{wellcovered-characterization2}, $G$ has no $4$-cycles. Now, if $G$ has a $3$-cycle then, by Theorem \ref{wellcovered-characterization2}, $G=T_{10}$. \smallskip

\noindent
{\rm (b)} By {\rm (a)}, the unique $3$-cycle is $(c_1,c_2,c_3,c_1)$ in $T_{10}$ and $deg_{T_{10}}(c_i)=3$ for $i=1,2,3$.  \smallskip
  
\noindent
{\rm (c)} By Figure \ref{specialgraphs}, $\tilde{e} \in \{ \tilde{e}_1,\tilde{e}_2,\tilde{e}_3\}$ and $G\in \{ T_{10},P_{10},P_{13},Q_{13}\}$, since $G\neq C_7$. Now, assume $\tilde{e}$ is in a $5$-cycle $C$. By Theorem \ref{wellcovered-characterization1}, $G\neq T_{10}$. If $G=P_{10}$, then $\tilde{e}_3$ is not in a $5$-cycle. Thus, $\tilde{e}\in \{ \tilde{e}_1,\tilde{e}_2 \}$ and $C\in \{ C_1,C_2\}$. Now, if $G=P_{13}$, then $\tilde{e}\in \{ \tilde{e}_1,\tilde{e}_2 \}$ and $C\in \{ C_1,C_2\}$. Finally, if $G=Q_{13}$, then $\tilde{e}_2,\tilde{e}_3$ are not in a $5$-cycle. Hence, $\tilde{e}=\tilde{e}_1$ and $C=C_1$. \smallskip

\noindent
{\rm (d)} By contradiction, suppose $G\neq C_7$. If $e\in E(P)$, then by {\rm (c)}, $e\in \{ e_1,e_2,e_3\}$. But, $e_i\cap e_j=\emptyset$ for $i\neq j$. A contradiction, since $P$ is a path. \qed 
\end{proof}

\begin{lemma}\label{Deg-x}
Let $G$ be a graph in $\{ C_7,T_{10},P_{10},P_{13},P_{14},Q_{13}\}$ with $I(D)$ unmixed. If $(z,y),(y,x)\in E(D)$ with $y\in V^{+}$ and $N_D(x)\setminus \{ y\} =\{ x_1\}$, then $deg_D(x_1)=2$.
\end{lemma}
\begin{proof}
By contradiction, suppose $deg_D(x_1)\geqslant 3$. Hence, there are $z_1,z_{1}^{\prime}\in N_D(x_1)\setminus \{ x\}$. By hypothesis, $deg_D(x)=2$. Then, by {\rm (b)} in Remark \ref{Not3-4-cycles}, $x$ is not in a $3$-cycle. So, $x_1\neq z$. Furthermore, by {\rm (a)} in Remark \ref{Not3-4-cycles}, $G$ has no $4$-cycles. Thus, $z\notin \{ z_{1}^{\prime}, z_1\}$ and $z_1,z_{1}^{\prime}\notin N_D(y)$. If $z_{1}^{\prime},z_1\in N_D(z)$, then $(x_1,z_{1}^{\prime},z,z_1,x_1)$ is a $4$-cycle. A contradiction, then we can assume $z_1\notin N_D(z)$. Consequently, $\{ x,z,z_1\}$ is a stable set, since $x$ is not in a $3$-cycle. A contradiction, by Lemma \ref{StableSet}, since $I(D)$ is unmixed.\qed  
\end{proof}

\begin{lemma}\label{Deg-Unm}
If $I(D)$ is unmixed, $G\in \{ C_7,T_{10},P_{10},P_{13},P_{14},Q_{13}\}$ and $\tilde{e}=(y,x)\in E(D)$ with $deg_D(x)=2$ and $y\in V^{+}$, then $G=P_{10}$ and $\tilde{e}=(d_i,b_j)$ with $\{ i,j\} =\{ 1,2\}$.
\end{lemma}
\begin{proof}
By Remark \ref{rem-V-R-P}, there is $z\in N_{D}^{-}(y)$. We set $N_D(x)=\{ y,x_1 \}$, then by {\rm (b)} in Remark \ref{Not3-4-cycles}, $z\neq x_1$. Thus, by Lemma \ref{Deg-x}, $deg_D(x_1)=2$. Now, we set $N_D(x_1)=\{ x,z_1\}$. By {\rm (b)} in Remark \ref{Not3-4-cycles}, $x$ is not in a $3$-cycle. So, $z,z_1\notin N_D(x)$ and $z_1\neq y$. Also, by {\rm (a)} in Remark \ref{Not3-4-cycles}, $G$ has no $4$-cycles. Then, $z_1\neq z$ and $z_1\notin N_D(y)$.  If $z\notin N_D(z_1)$, then $\{ x,z,z_1\}$ is a stable set. A contradiction, by Lemma \ref{StableSet}, since $I(D)$ is unmixed. Hence, $\{ z_1,z\} \in E(G)$ and $C:=(z,y,x,x_1,z_1,z)$ is a $5$-cycle. Suppose $y^{\prime}\in N_{D}^{-}(y)\setminus \{ z\}$, then $y^{\prime}\notin N_D(z_1)$, since $(z_1,z,y,y^{\prime},z_1)$ is not in a $4$-cycle in $G$. Consequently, $\{ y^{\prime},x,z_1\} $ is a stable set, since $deg_D(x)=2$. A contradiction, by Lemma \ref{StableSet}, since $(y^{\prime},y),(y,x)\in E(D)$, $y\in V^{+}$, $N_D(x)=\{ y,x_1\}$ and $z_1\in N_D(x_1)\setminus N_{D}^{+}(y)$. Hence, $N_{D}^{-}(y)=\{ z\}$. Now, by {\rm (c)} in Remark \ref{Not3-4-cycles} and by symmetry of $P_{10}$ and $P_{13}$, we can assume $\{ x,x_1\} =\tilde{e}_1$, $C=C_1$ and $G\in \{ P_{10},P_{13},Q_{13}\}$. \smallskip

\noindent
First, assume $G=P_{13}$. By symmetry and notation of Figure \ref{specialgraphs}, we can suppose $x_1=a_1$ and $x=a_2$, since $\tilde{e}_1=\{ x,x_1\}$. Then, $y=b_2$, $z=c_1$ and $z_1=b_1$. Thus, $(b_2,d_2)\in E(D)$, since $y=b_2$ and $N_{D}^{-}(y)=\{ z\} =\{ c_1\}$. By Figure \ref{specialgraphs}, $(c_1,b_2)=(z,y)\in E(D)$, $(b_2,d_2)\in E(D)$, $b_2=y\in V^{+}$ and $N_D(d_2)=\{ b_2, b_4,v\}$. Also, $a_4\in N_D(b_4)$, $d_1\in N_D(v)\setminus N_D(b_2)$ and $\{ c_1,d_2,a_4,d_1\}$ is a stable set. A contradiction, by Lemma \ref{StableSet}, since $I(D)$ \linebreak is unmixed. \smallskip

\noindent
Now, suppose $G=Q_{13}$. By the symmetry of we can suppose $x=a_2$ and $x_1=a_1$, then $d_2=y\in V^{+}$, $z=h$ and $(h,d_2)=(z,y)\in E(D)$. So, $(d_2,c_2)\in E(D)$, since $N_{D}^{-}(d_2)=N_{D}^{-}(y)=\{ z\} =\{ h\}$. A contradiction, by Lemma \ref{StableSet}, since $(h,d_2),(d_2,c_2)\in E(D)$, $d_2\in V^{+}$, $N_D(c_2)=\{ d_2,b_1\}$, $g_1\in N_D(b_1)\setminus N_{D}^{+}(d_2)$ and $\{ h,c_2,g_1\}$ is a stable set.  \smallskip

\noindent
Hence, $G=P_{10}$ and $\{ x,x_1 \} =\tilde{e}_1=\{ a_1,b_1\}$. If $x=a_1$ and $x_1=b_1$, then $g_1=y\in V^{+}$, $(d_1,g_1)=(z,y)\in E(D)$, since $C=C_1$. Furthermore, $(g_1,c_1)\in E(D)$, since $N_{D}^{-}(g_1)=N_{D}^{-}(y)=\{ z\} =\{ d_1\}$. A contradiction by Lemma \ref{StableSet}, since $(d_1,g_1),(g_1,c_1)\in E(D)$, $g_1\in V^{+}$, $N_D(c_1)=\{ g_1,c_2\}$, $g_2\in N_D(c_2)$ and $\{ d_1,c_1,g_2\} $ is stable. Therefore, $x=b_1$ and $x_1=a_1$, implying $y=d_2$ and $(y,x)=(d_{2},b_{1})$, since $C=C_1$.  \qed
\end{proof}

\begin{remark}\rm\label{V2}
Assume $I(D)$ is unmixed, $G\in \{ C_7,T_{10},Q_{13},P_{13},P_{14}\}$, $\mathcal{C}$ is a strong vertex cover of $D$ and $y\in \mathcal{C}\cap V^{+}$ such that $N_G(y)\setminus \mathcal{C} \subseteq V_2:=\{ a\in V(G)\mid deg_G(a)=2\}$. We take $b\in N_G(y)\setminus \mathcal{C}$. If $(y,b)\in E(D)$, then by Lemma \ref{Deg-Unm}, $G=P_{10}$. A contradiction, then $(b,y)\in E(D)$. Consequently, $N_D(y)\setminus \mathcal{C} \subseteq N_{D}^{-}(y)$, i.e. $N_{D}^{+}(y)\subseteq \mathcal{C}$. Hence, $y\in \mathcal{C} \setminus L_1(\mathcal{C})$.
\end{remark}

\begin{proposition}\label{T10,P13,P14,Q13-Unmix}
If $I(D)$ is unmixed, with $G\in \{ C_7,T_{10},Q_{13},P_{13},P_{14}\}$, then the vertices of $V^{+}$ are sinks.
\end{proposition}
\begin{proof}
By contradiction, suppose there is $(y,x)\in E(D)$ with $y\in V^{+}$. Then, by Lemma \ref{Deg-Unm}, $deg_D(x)\geqslant 3$. Thus, $G\neq C_7$. By Remark \ref{rem-V-R-P}, $y$ is not a source. So, there is $(z,y)\in E(D)$. We set $V_2:=\{ a\in V(G)\mid deg_G(a)=2\}$. By Theorem \ref{theorem42}, $L_3(\tilde{\mathcal{C}})=\emptyset$ for each strong vertex cover $\tilde{\mathcal{C}}$ of $D$, since $I(D)$ is unmixed. Hence, to obtain a contradiction, we will give a vertex cover $\mathcal{C}$ of $D$ such that $L_3(\mathcal{C})=\{ x\}$ and $y\in \mathcal{C}\setminus L_1(\mathcal{C})$, since with these conditions $\mathcal{C}$ is strong. We will use the notation of Figure \ref{specialgraphs}. \smallskip 

\noindent
{\bf Case (1)} If $D=T_{10}$, then $x\in \{ c_1,c_2,c_3,v\}$, since $deg_G(x)\geqslant 3$. 

\noindent
\underline{Case (1.a)} $x\in \{ c_1,c_2,c_3\}$. By symmetry of $P_{10}$, we can assume $x=c_1$ and $y\in \{ b_1,c_2\}$. Thus, $\mathcal{C}_1=\{ v,a_2,a_3,b_1,c_1,c_2,c_3\}$ is a vertex cover with $L_3(\mathcal{C}_1)=\{ x\}$. If $y=b_1$, then $y\in C_1$ and $z=a_1$, since $N_D(b_1)=\{ a_1,c_1\}$. So, $N_{D}^{+}(y)=\{ c_1\} \subseteq \mathcal{C}_1$. Now, if $y=c_2$, then $y\in \mathcal{C}_1$. Furthermore, by Lemma \ref{Deg-Unm}, $(b_2,c_2)\in E(D)$, since $c_2=y\in V^{+}$, $b_2\in V_2$ and $I(D)$ is unmixed. Consequently, $N_{D}^{+}(y)\subseteq \{ c_1,c_3\} \subseteq \mathcal{C}_1$. Hence, $y\in \mathcal{C}_1\setminus L_1(\mathcal{C}_1)$ and we take $\mathcal{C}=\mathcal{C}_1$. 

\noindent
\underline{Case (1.b)} $x=v$. Then, $\mathcal{C}_2=\{ v,a_1,a_2,a_3,c_1,c_2,c_3\}$ is a vertex cover with $L_3(\mathcal{C}_2)=\{ x\}$. By symmetry of $P_{10}$, we can suppose $y=a_1$. Consequently, $z=b_1$ and $N_{D}^{+}(y)=\{ v\} \subseteq \mathcal{C}_2$, since $a_1\in V_2$. Hence, $y\in \mathcal{C}_2\setminus L_1(\mathcal{C}_2)$ and we take $\mathcal{C}=\mathcal{C}_2$. \medskip

\noindent
{\bf Case (2)} If $D=P_{14}$, then, by symmetry, we can assume $y=a_1$ and $x\in \{ a_2,b_1\}$. 

\noindent
\underline{Case (2.a)} $x=a_2$. Thus, $z\in \{ a_7,b_1\}$. We take $\mathcal{C}_3=\{ a_1,a_2,a_3,a_4,a_6,b_1,b_2,b_5,b_6,b_7\}$ if $z=a_7$ or $\mathcal{C}_3=\{ a_1,a_2,a_3,a_5,a_7,b_2,b_3,b_4,b_5,b_6\}$ if $z=b_1$. So, $\mathcal{C}_3$ is a vertex cover of $D$, $y\in \mathcal{C}_3$ and $L_3(\mathcal{C}_3)=\{ x\}$. Furthermore, $N_D(y)\setminus \mathcal{C}_3=\{ z\}$ and $z\in N_{D}^{-}(y)$, then $y\in \mathcal{C}_3\setminus L_1(\mathcal{C}_3)$ and we take $\mathcal{C}=\mathcal{C}_3$.  

\noindent
\underline{Case (2.b)} $x=b_1$. By symmetry of $P_{14}$, we can suppose $z=a_2$. Then, \linebreak $\mathcal{C}_4=\{ a_1,a_3,a_4,a_5,a_7,b_1,b_2,b_3,b_6,b_7\}$ is a vertex cover of $D$ with $L_3(\mathcal{C}_4)=\{ x\}$. Also, $N_{D}(y)=\setminus\mathcal{C}_4=\{ a_2\}$ and $a_2=z\in N_{D}^{-}(y)$. Hence, $y \in \mathcal{C}_4\setminus L_1(\mathcal{C}_4)$ and we take $\mathcal{C}=\mathcal{C}_4$. \smallskip

\noindent
{\bf Case (3)} If $D=P_{13}$, then we can assume $x\in \{ b_1,c_2,d_1\}$, since $deg_G(x)\geqslant 3$. 

\noindent
\underline{Case (3.a)} $x=c_2$. Then, $y\in N_D(c_2)=\{ b_3,b_4,c_1\}$. Without loss of generality, we can su\-ppo\-se $y\in \{ b_3,c_1\}$. If $y=c_1$, then $z\in \{ b_1,b_2\}$. By symmetry, we can assume $z=b_1$ so $N_{D}^{+}(y)=N_{D}^{+}(c_1)\subseteq\{ b_2,c_2\}$. We take $\mathcal{C}_5=\{ a_1,a_4,b_2,b_3,b_4,c_1,c_2,d_1,v\}$. Thus, $\mathcal{C}_5$ is a vertex cover of $D$ with $L_3(\mathcal{C}_5)=\{ x\}$, $y\in \mathcal{C}_5$ and $N_{D}^{+}(y)\subseteq \{b_2,c_2\} \subseteq \mathcal{C}_5$. Now, if $y=b_3$, then $b_3\in V^{+}\cap \mathcal{C}_5$. Hence, by Lemma \ref{Deg-Unm}, $(a_3,b_3)\in E(D)$, sin\-ce $a_3\in V_2$ and $I(D)$ is unmixed. Consequently, $N_{D}^{+}(y)\subseteq \{ c_2,d_1\} \subset \mathcal{C}_5$. Therefore, $y\in \mathcal{C}_5\setminus L_1(\mathcal{C}_5)$ and we take $\mathcal{C}=\mathcal{C}_5$.  

\noindent
\underline{Case (3.b)} $x=b_1$. Hence, $\mathcal{C}_6=\{ a_1,a_3,b_1,b_2,b_3,b_4,c_1,d_1,d_2\}$ is a vertex cover of $D$ with $L_3(\mathcal{C}_6)=\{ x\}$ and $y\in N_D(b_1)=\{ a_1,c_1,d_1\}$. Also, $y \in \mathcal{C}_6$. If $y\in \{ a_1,d_1\}$, then $N_D(y)\setminus \mathcal{C}_6\subseteq \{ a_2,v\} \subseteq V_2$. Consequently, by Lemma \ref{Deg-Unm}, $(b,y)\in E(D)$ for each $b\in N_D(y)\setminus \mathcal{C}_6$, since $y\in V^{+}$. So, $N_{D}^{+}\subseteq \mathcal{C}_6$ implying $y\in \mathcal{C}_6\setminus L_1(\mathcal{C}_6)$. Now, if $y=c_1$. We can assume $(c_2,c_1)\in E(D)$, since in another case we have the case (3.a) with $x=c_2$ and $y=c_1$. Thus, $y=c_1\in \mathcal{C}_6\setminus L_1(\mathcal{C}_6)$, since $N_{D}^{+}(c_1)\subseteq\{ b_1,b_2\}\subset \mathcal{C}_6$. Therefore, we take $\mathcal{C}=\mathcal{C}_6$.

\noindent
\underline{Case (3.c)} $x=d_1$. Then, $y\in N_G(d_1)=\{ b_1,b_3,v\}$. By symmetry, we can assume $y\in \{ b_1,v\}$. Furthermore, $\mathcal{C}_7=\{ a_2,a_4,b_1,b_2,b_3,b_4,c_1,d_1,v\}$ is a vertex cover of $D$ with $L_3(\mathcal{C}_7)=\{ x\}$ and $y\in \mathcal{C}_7$. If $y=b_1$, then $N_D(y)\setminus \mathcal{C}_7 \subseteq \{a_1\}$. Also, by Lemma \ref{Deg-Unm}, $N_D(y)\setminus \mathcal{C}_7\subseteq N_{D}^{-}(y)$, since $a_1\in V_2$ and $y\in V^{+}$. Thus, $y\in \mathcal{C}_7\setminus L_1(\mathcal{C}_7)$. Now, if $y=v$, then $z=d_2$, since $N_D(v)=\{ d_1,d_2\}$. This implies $N_{D}^{+}(v)=\{ d_1\} \subseteq \mathcal{C}_7$. So, $y\in \mathcal{C}_7\setminus L_1(\mathcal{C}_7)$ and we take $\mathcal{C}=\mathcal{C}_7$. \smallskip

\noindent
{\bf Case (4)} $D=Q_{13}$. Hence, $x\in \{ d_1,d_2,g_1,g_2,h,h^{\prime}\}$, since $deg_D(x)\geqslant 3$. By symmetry, we can suppose $x\in \{ d_2,g_2,h,h^{\prime}\}$.

\noindent
\underline{Case (4.a)} $x\in \{ d_2,g_2\}$. We take $\mathcal{C}_8=\{ a_2,c_1,c_2,d_1,d_2,g_1,g_2,h,h^{\prime}\}$ if $x=d_2$ or $\mathcal{C}_8=\{ a_1,b_2,c_2,d_1,d_2,g_1,g_2,h,h^{\prime}\}$ if 
$x=g_2$. Thus, $\mathcal{C}_8$ is a vertex cover of $D$ with $L_3(\mathcal{C}_8)=\{ x\}$ and $V(G)\setminus \mathcal{C}_8=\{ a_1,b_1,b_2,v\} \cup \{ a_2,b_1,c_1,v\} \subseteq V_2$. Consequently, by Lemma \ref{Deg-Unm}, $N_D(y)\setminus \mathcal{C}_8\subseteq N_{D}^{-}(y)$, since $y\in V^{+}$. This implies $N_{D}^{+}(y)\subseteq \mathcal{C}_8$. Therefore $y\in \mathcal{C}_8\setminus L_1(\mathcal{C}_8)$ and we take $\mathcal{C}=\mathcal{C}_8$. 

\noindent
\underline{Case (4.b)} $x\in \{ h,h^{\prime}\}$. We take $\mathcal{C}_9=\{ a_2,b_1,b_2,d_1,d_2,g_1,g_2,h,v\}$ if $x=h$ or $\mathcal{C}_9=\{ a_2,c_1,c_2,d_1,d_2,g_1,g_2,h^{\prime},v\}$ if $x=h^{\prime}$. Thus, $\mathcal{C}_9$ is a vertex cover of $D$ with $L_3(\mathcal{C}_9)=\{ x\}$. Also, $y\in N_G(x)\subseteq \{ v,d_1,d_2,g_1,g_2\}$. If $y=v$, then $\{ x,z\} \subseteq N_D(y)=\{ h,h^{\prime}\}$. Hence, $N_{D}^{+}(y)=\{ x\} \subseteq \mathcal{C}_9$, then $y\in \mathcal{C}_9\setminus L_1(\mathcal{C}_9)$. Now, if $y\neq v$, then $y\in \{ d_1,d_2\}$ when $x=h$ or $y\in \{ g_1,g_2\}$ when $x=h^{\prime}$. Then, $N_D(y)\setminus \mathcal{C}_9\subseteq \{ a_1,c_1,c_2\} \subseteq V_2$ if $x=h$ or $N_D(y)\setminus \mathcal{C}_9\subseteq \{ b_1,b_2\} \subseteq V_2$ if $x=h^{\prime}$. Consequently, by Lemma \ref{Deg-Unm}, $N_D(y)\setminus \mathcal{C}_9\subseteq N_{D}^{-}(y)$, since $y\in V^{+}$. Thus, $N_{D}^{+}(y)\subseteq \mathcal{C}_9$ and $y\in \mathcal{C}_9\setminus L_1(\mathcal{C}_9)$. Therefore, we take $\mathcal{C}=\mathcal{C}_9$.  \qed
\end{proof}

\begin{theorem}\label{No4,5Cyc-Unmix}
Let $D=(G,\mathcal{O},w)$ be a connected weighted oriented graph without $4$- and $5$-cycles. Hence, $I(D)$ is unmixed if and only if $D$ satisfies one of the following conditions:
\begin{enumerate}[noitemsep]
\item[{\rm (a)}] $G\in \{C_7,T_{10}\}$ and the vertices of $V^{+}$ are sinks.
\item[{\rm (b)}] Each vertex is in exactly one simplex of $D$ and each simplex of $D$ has no generating $\star$-semi-forests.
\end{enumerate}
\end{theorem}
\begin{proof}
$\Rightarrow )$ By {\rm (3)} in Theorem \ref{theorem42} and Remark \ref{1star}, $G$ is well-covered. Thus, by Theorem \ref{wellcovered-characterization1}, $G\in \{ C_7,T_{10}\}$ or $\{ V(H)\mid H\in S_G\}$ is a partition of $V(G)$. If $G\in \{ C_7,T_{10}\}$, then by Proposition \ref{T10,P13,P14,Q13-Unmix}, $D$ satisfies {\rm (a)}. Now, if $\{ V(H)\mid H\in S_G\}$ is a partition of $V(G)$, then $G$ is an $SCQ$ graph with $C_G=Q_G=\emptyset$. Hence, by Theorem \ref{SCQ-char}, $D$ satisfies {\rm (b)}. \medskip

\noindent
$\Leftarrow )$ If $D$ satisfies {\rm (a)}, then by Theorem \ref{wellcovered-characterization1}, $G$ is well-covered. Consequently, by Corollary \ref{cor-oct30}, $I(D)$ is unmixed. Now, if $D$ satisfies {\rm (b)}, then $G$ is an $SCQ$ graph, with $C_G=Q_G=\emptyset$. Therefore, by {\rm (b)} and Theorem \ref{SCQ-char}, $I(D)$ is unmixed.  \qed
\end{proof}

\begin{corollary}
Let $D$ be a weighted oriented graph without isolated vertices and \linebreak $girth (G)\geqslant 6$. Hence, $I(D)$ is unmixed if and only if $D$ satisfies one of following pro\-per\-ties:
\begin{enumerate}[noitemsep]
\item[{\rm (a)}] $G=C_7$ and the vertices of $V^{+}$ are sinks.
\item[{\rm (b)}] $G$ has a perfect matching $e_1=\{ x_1,x_{1}^{\prime}\} ,\ldots ,e_r=\{ x_r,x_{r}^{\prime}\}$ where $deg_D(x_{1}^{\prime})=\cdots =deg_D(x_{r}^{\prime})=1$ and $(x_{j}^{\prime},x_j)\in E(D)$ if $x_j\in V^{+}$.
\end{enumerate}
\end{corollary}
\begin{proof}
$\Leftarrow )$ If $D$ satisfies {\rm (a)}, then $I(D)$ is unmixed by {\rm (a)} in Theorem \ref{No4,5Cyc-Unmix}. Now, assume $D$ satisfies {\rm (b)}. Assume $b^{\prime}\in N_{D}^{+}(a)$ with $a\in V^{+}$ such that $\{b,b^{\prime}\}=e_j=\{ x_j,x_{j}^{\prime}\}$ for some $1\leqslant j\leqslant r$. If $b^{\prime}=x_{j}^{\prime}$, then $x_j=a\in V^{+}$, since $deg_D(x_{j}^{\prime})=1$. So, $x_{j}^{\prime}=b^{\prime}\in N_{D}^{+}(a)=N_{D}^{+}(x_j)$. But by hypothesis, $(x_{j}^{\prime},x_j)\in E(D)$, since $x_j\in V^{+}$. A contradiction, then $b^{\prime}=x_j$. Thus, $b=x_{j}^{\prime}$ implies $N_D(b)=\{ x_j\} =\{b^{\prime}\} \subseteq N_{D}^{+}(a)$. Hence, by Proposition \ref{prop-unmixed}, $|\mathcal{C}\cap e_j|=1$ where $\mathcal{C}$ is a strong vertex cover. So, $|\mathcal{C}|=r$. Therefore, by {\rm (2)} in Theorem \ref{theorem42}, $I(D)$ is unmixed. \medskip

\noindent
$\Rightarrow )$ $G\neq T_{10}$, since $T_{10}$ has $3$-cycles. Hence, by Theorem \ref{No4,5Cyc-Unmix}, $D$ satisfies {\rm (a)} or $\{ V(H)\mid H\in S_G\}$ is a partition of $V(G)$. Furthermore, $S_G\subseteq E(G)$, since $G$ has not isolated vertices and $girth (G)\geqslant 6$. Thus, we can assume $S_G=\{ e_1,\ldots ,e_r\}$ where $e_i=\{ x_i,x_{i}^{\prime}\}$ and $deg_D(x_{i}^{\prime})=1$ for $i=1,\ldots ,s$. Also, by Theorem \ref{No4,5Cyc-Unmix}, each $e_i$ has no generating $\star$-semi-forests. So, by Theorem \ref{theorem-oct29}, $e_i\not\subseteq \mathcal{C}$ for each strong vertex cover $\mathcal{C}$. Consequently, $|e_i\cap \mathcal{C}|=1$, since $\mathcal{C}$ is a vertex cover. Then, $e_i$ satisfies {\rm (2)} of Proposition \ref{prop-unmixed}. Now, suppose $(x_j,x_{j}^{\prime})\in E(D)$ and $x_j\in V^{+}$. We take $a=b=x_j$ and $b^{\prime}=x_{j}^{\prime}$, then by {\rm (2)} of Proposition \ref{prop-unmixed}, $N_D(x_j)=N_D(b)\subseteq N_{D}^{+}(a)=N_{D}^{+}(x_j)$. Hence, $x_j$ is a source. A contradiction, by Remark \ref{rem-V-R-P}, since $x_j\in V^{+}$. Therefore, $D$ satisfies {\rm (b)}.  \qed
\end{proof}

\begin{proposition}\label{P10-Unmix}
If $G=P_{10}$, then the following properties are equivalent:
\begin{enumerate}[noitemsep]
\item[{\rm (1)}] $I(D)$ is unmixed.
\item[{\rm (2)}] If $y\in V^{+}$ and $y$ is not a sink, then $y=d_1$ with $N_{D}^{+}(y)=\{ g_1,b_2\}$ or $y=d_2$ with $N_{D}^{+}(y)=\{ g_2,b_1\}$.
\end{enumerate}
\end{proposition}
\begin{proof}
${\rm (2)} \Rightarrow {\rm (1)}$ Let $\mathcal{C}$ be a strong vertex cover of $D$. Suppose $x\in L_3(\mathcal{C})$. Then, there is $y\in \big( \mathcal{C}\setminus L_1(\mathcal{C})\big) \cap V^{+}$ such that $x\in N_{D}^{+}(y)$. Thus, by {\rm (2)}, $y\in \{d_1,d_2\}$ and $x\in N_{D}^{+}(y)\subseteq \{ b_1,b_2,g_1,g_2\}$. Hence, $L_3(\mathcal{C})\subseteq \{ b_1,b_2,g_1,g_2\}$. By symmetry of $P_{10}$, we can assume $y=d_1$. Then by {\rm (2)}, $x\in N_{D}^{+}(y)= \{ g_1,b_2\}$. Also, $\{ g_1,d_1,b_2\} =N_{D}^{+}[y]\subseteq \mathcal{C}$, since $y\in \mathcal{C}\setminus L_1(\mathcal{C})$. But $y=d_1\notin L_3(\mathcal{C})$, then $N_D(y)\not\subseteq \mathcal{C}$. Thus, $d_2\notin \mathcal{C}$. So, by Remark \ref{einC}, $\{ b_1,d_1,g_2\} =N_D(d_2)\subseteq \mathcal{C}$. Furthermore, $\{ a_1,a_2\} \cap L_3(\mathcal{C})=\emptyset$, since $L_3(\mathcal{C})\subseteq \{ b_1,b_2,g_1,g_2\}$. Consequently, $\{ a_1,a_2\} \cap \mathcal{C}=\emptyset$, since $N_D(a_1,a_2)=\{ g_1,b_1,g_2,b_2\} \subseteq \mathcal{C}$. But, $x\in \{ g_1,b_2\} \cap L_3(\mathcal{C})$, then $a_1\in N_D(g_1)\subseteq \mathcal{C}$ or $a_2\in N_D(b_2)\subseteq \mathcal{C}$. Hence, $\{ a_1,a_2\} \cap \mathcal{C}\neq \emptyset$. A contradiction, then $L_3(\mathcal{C})=\emptyset$. Also, by Theorem \ref{wellcovered-characterization1} and Remark \ref{1star}, $I(G)$ is unmixed. Therefore, by {\rm (3)} in Theorem \ref{theorem42}, $I(D)$ is unmixed. \medskip

\noindent
${\rm (1)} \Rightarrow {\rm (2)}$ We take $V_2:=\{ a\in V(G)\mid deg_G(a)=2\}$ and $y\in V^{+}$, such that $y$ is not a sink, then there is $(y,x)\in E(D)$. Also, by Remark \ref{rem-V-R-P}, there is $z\in N_{D}^{-}(y)$. We will prove $y\in \{ d_1,d_2\}$. If $deg_D(x)=2$, then by Lemma \ref{Deg-Unm}, $y\in \{ d_1,d_2\}$. Now, we assume $deg_D(x)\geqslant 3$, then by symmetry of $P_{10}$, we can suppose $x\in \{ g_1,d_1\}$. 

\noindent
First suppose $x=g_1$. Thus, $y\in N_D(g_1)=\{ a_1,c_1,d_1\}$. If $y\in \{ a_1,c_1\}$, then $deg_D(y)=2$ and $y\in \mathcal{C}_1=\{ a_1,a_2,c_1,d_1,d_2,g_1,g_2\}$ is a vertex cover of $D$ with $L_3(\mathcal{C}_1)=\{ x\}$. Hence, $N_D(y)=\{ x,z\}$ implies $N_{D}^{+}(y)=\{ x\} =\{ g_1\} \subseteq \mathcal{C}_1$. Consequently, $y\in \mathcal{C}_1\setminus L_1(\mathcal{C}_1)$. Hence, $\mathcal{C}_1$ is strong, since $L_3(\mathcal{C}_1)=\{ x\}$. A contradiction, by Theorem \ref{theorem42}, then $y=d_1$. 

\noindent
Now, suppose $x=d_1$. Then, $\mathcal{C}_2=\{ a_1,b_2,c_2,d_1,d_2,g_1,g_2\}$  is a vertex cover of $D$ with $L_3(\mathcal{C}_2)=\{ x\}$. Also, $y\in N_D(x)=\{ g_1,b_2,d_2\}$. If $y\in \{ g_1,b_2\}$, then $N_D(y)\setminus \{ d_1\} \subseteq N_D(g_1,b_2)\setminus \{ d_1\} =\{ a_1,a_2,c_1\} \subseteq V_2$. Hence, by Lemma \ref{Deg-Unm}, $N_D(y)\setminus \{d_1\} \subseteq N_{D}^{-}(y)$, since $y\notin \{ b_1,b_2\}$. Thus, $N_{D}^{+}(y)=\{ d_1\} =\{ x\} \subseteq \mathcal{C}_2$ implying $y\in \mathcal{C}_2\setminus L_1(\mathcal{C}_2)$. Consequently, $\mathcal{C}_2$ is strong with $L_3(\mathcal{C}_2)\neq \emptyset$. A contradiction, by Theorem \ref{theorem42}, then $y=d_2$. 

\noindent
Therefore $y\in \{ d_1,d_2\}$. By symmetry of $P_{10}$, we can assume $y=d_1$. Now, we will prove $N_{D}^{+}(y)=\{ g_1,b_2\}$. By contradiction, in each one of the following cases, we give a strong vertex cover $\mathcal{C}^{\prime}$ with $L_3(\mathcal{C}^{\prime})\neq \emptyset$, since $I(D)$ is unmixed and $z\in N_{D}^{-}(y)$. \smallskip

\noindent
{\bf Case (1)} $g_1\notin N_{D}^{+}(d_1)$ and $b_2\in N_{D}^{+}(d_1)$. So, $N_{D}^{+}(d_1)\subseteq \{ b_2,d_2\}$ and $\mathcal{C}_{1}^{\prime}=$ \linebreak $\{ a_1,a_2,b_2,c_1,c_2,d_1,d_2\}$ is a vertex cover of $D$ with $L_3(\mathcal{C}_{1}^{\prime})=\{ b_2\}$. Also, $(d_1,b_2)\in E(D)$ and $y=d_1\in \big(\mathcal{C}_{1}^{\prime}\setminus L_1(\mathcal{C}_{1}^{\prime})\big) \cap V^{+}$, since $b_2\in N_{D}^{+}[d_1]\subseteq \{ d_1,b_2,d_2\} \subset \mathcal{C}_{1}^{\prime}$. Hence, $\mathcal{C}_{1}^{\prime}$ is a strong vertex cover of $D$. \smallskip 

\noindent
{\bf Case (2)} $b_2\notin N_{D}^{+}(d_1)$ and $g_1\in N_{D}^{+}(d_1)$. Then, $N_{D}^{+}(d_1)\subseteq \{ g_1,d_2\}$ and $\mathcal{C}_{2}^{\prime}=$ \linebreak $\{ a_1,a_2,c_1,d_1,d_2,g_1,g_2\}$ is a vertex cover of $D$ with $L_3(\mathcal{C}_{2}^{\prime})=\{ g_1\}$. Furthermore $(d_1,g_1)\in E(D)$ and $d_1=y \in \big(\mathcal{C}_{2}^{\prime}\setminus L_1(\mathcal{C}_{2}^{\prime})\big) \cap V^{+}$, since $g_1\in N_{D}^{+}[d_1]\subseteq \{ d_1,g_1,d_2\}\subset \mathcal{C}_{2}^{\prime}$. Hence, $\mathcal{C}_{2}^{\prime}$ is a strong vertex cover of $D$. \smallskip

\noindent
{\bf Case (3)} $b_2,g_1\notin N_{D}^{+}(d_1)$. Thus, $z=d_1$, $N_{D}^{+}(d_1)=\{ d_2\}$ and $\mathcal{C}_{3}^{\prime}=$ \linebreak $\{ a_2,b_1,c_1,d_1,d_2,g_1,g_2\}$ is a vertex cover of $D$ with $L_6(\mathcal{C}_{3}^{\prime})=\{ d_2\}$. Also, $(d_1,d_2)\in E(D)$ and $d_1=y\in \big(\mathcal{C}_{3}^{\prime}\setminus L_1(\mathcal{C}_{3}^{\prime})\big) \cap V^{+}$, since $N_{D}^{+}[d_1]=\{ d_1,d_2\}\subset \mathcal{C}_{3}^{\prime}$. Hence, $\mathcal{C}_{3}^{\prime}$ is a strong vertex cover of $D$. \qed 
\end{proof}

\begin{theorem}\label{No3,4,5Cyc-Unmix}
Let $D=(G,\mathcal{O},w)$ be a connected weighted oriented graph, with \linebreak ${\rm girth}(G)\geqslant 5$. Hence, $I(D)$ is unmixed if and only if $D$ satisfies one of the following properties:
\begin{enumerate}[noitemsep]
\item[{\rm (a)}] $G\in \{K_1,C_7,Q_{13},P_{13},P_{14}\}$ and the vertices of $V^{+}$ are sinks.
\item[{\rm (b)}] $G=P_{10}$, furthermore if $y$ is not a sink in $V^{+}$, then $y=d_1$ with $N_{D}^{+}(y)=\{ g_1,b_2\}$ or $y=d_2$ with $N_{D}^{+}(y)=\{ g_2,b_1\}$.
\item[{\rm (c)}] Each vertex is in exactly one simplex of $G$ or in exactly one basic $5$-cycle of $G$. Furthermore, each simplex of $D$ has not a generating $\star$-semi-forest and each basic $5$-cycle of $D$ has the $\star$-property.
\end{enumerate}
\end{theorem}
\begin{proof}
$\Rightarrow )$ By {\rm (3)} in Theorem \ref{theorem42} and Remark \ref{1star}, $G$ is well-covered. By Theorem \ref{wellcovered-characterization2}, $G\in \{ K_1,C_7,P_{10},P_{13},P_{14},Q_{13}\}$ or $\{ V(H)\mid H\in S_G\cup C_G\}$ is a partition of $V(G)$. If $\{ V(H)\mid H\in S_G\cup C_G\}$ is a partition of $V(G)$, then $G$ is an $SCQ$ graph with $Q_G=\emptyset$. Hence, by Theorem \ref{SCQ-char}, $D$ satisfies {\rm (c)}. Now, if $G=P_{10}$, then by Proposition \ref{P10-Unmix}, $D$ satisfies {\rm (b)}. Furthermore, if $G\in \{ C_7,Q_{13},P_{13},P_{14}\}$, then by Proposition \ref{T10,P13,P14,Q13-Unmix}, $D$ sa\-tis\-fies {\rm (a)}. Finally, if $G=K_1$, then by Remark \ref{rem-V-R-P}, $V^{+}=\emptyset$. \medskip

\noindent
$\Leftarrow )$ If $D$ satisfies {\rm (b)}, then by Proposition \ref{P10-Unmix}, $I(D)$ is unmixed. Now, if $D$ satisfies {\rm (c)}, then $G$ is an $SCQ$ graph with $Q_G=\emptyset$. Consequently, by Theorem \ref{SCQ-char}, $I(D)$ is unmixed. Finally, if $D$ satisfies {\rm (a)}, then by Theorem \ref{wellcovered-characterization2}, $G$ is well-covered. Therefore, by Corollary \ref{cor-oct30}, $I(D)$ is unmixed.   \qed
\end{proof}

\begin{remark}\rm
A graph is well-covered if and only if each connected component is well-covered. Hence, when $D$ is no connected in Theorem \ref{No4,5Cyc-Unmix} (resp. \ref{No3,4,5Cyc-Unmix}), $I(D)$ is unmixed if and only if each connected component of $D$ satisfies {\rm (a)} or {\rm (b)} (resp. {\rm (a)}, {\rm (b)} or {\rm (c)}). 
\end{remark}

\section{Examples}
 
\begin{example}\rm
Let $D_1$ be the following weighted oriented graph. 
\end{example}
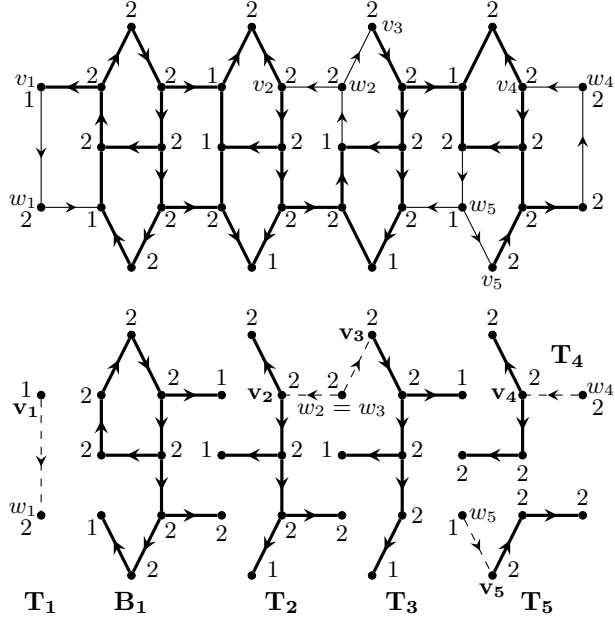
\begin{minipage}[c]{8cm} 
\centering
\begin{tikzpicture}[dot/.style={draw,fill,circle,inner sep=1pt},scale=0.8]

%%%%%%%%%%%%%%%%%%%%%%%%%%%%%%%%%%%%%%%%%%%%%%%

\node[draw,fill,circle,inner sep=1pt] (1) at (-7,1){};
\node (101) at (-7.25,1.1){{\small $v_{\small 1}$}};
\node (201) at (-7.2,0.8){{\small $1$}};
\node[draw,fill,circle,inner sep=1pt](2) at (-7,-1) {};
\node (102) at (-7.3,-.9){{\small $w_{\small 1}$}};
\node (202) at (-7.25,-1.2){{\small $2$}};

\path [draw,postaction={on each segment={mid arrow=thick}}] 
(1) to (2);

%%%%%%%%%%%%%%%%%%%%%%%%%%%%%%%%%%%%%%%%%

\node[draw,fill,circle,inner sep=1pt] (8) at (-5,1){};
\node (308) at (-4.85,1.2){{\small $2$}};
\node[draw,fill,circle,inner sep=1pt] (9) at (-5,0) {};
\node (309) at (-4.75,.1){{\small $2$}};
\node[draw,fill,circle,inner sep=1pt](5) at (-6,-1) {};
\node (305) at (-6.15,-1.2){{\small $1$}};
\node[draw,fill,circle,inner sep=1pt] (10) at (-5,-1){};
\node (310) at (-4.85,-1.2){{\small $2$}};
\node[draw,fill,circle,inner sep=1pt] (3) at (-6,1){};
\node (303) at (-6.15,1.2){{\small $2$}};
\node[draw,fill,circle,inner sep=1pt] (4) at (-6,0) {};
\node (304) at (-6.25,.1){{\small $2$}};
\node[draw,fill,circle,inner sep=1pt] (11) at (-4,1){};
\node (311) at (-4.15,1.2){{\small $1$}};;\node[draw,fill,circle,inner sep=1pt] (12) at (-4,0) {};
\node (312) at (-4.25,.1){{\small $1$}};
\node[draw,fill,circle,inner sep=1pt] (13) at (-4,-1){};
\node (313) at (-4.15,-1.2){{\small $2$}};
\node[draw,fill,circle,inner sep=1pt] (7) at (-5.5,-2){};
\node (307) at (-5.15,-1.9){{\small $2$}};
\node[draw,fill,circle,inner sep=1pt] (6) at (-5.5,2){};
\node (306) at (-5.5,2.3){{\small $2$}};

\path [draw,postaction={on each segment={mid arrow=thick}}] (2) to (5);

\path [draw,very thick,postaction={on each segment={mid arrow=very thick}}]
(3) to (1)
(6) to (8)
(8) to (9)
(10) to (13)
(8) to (11)
(3) to (6)
(4) to (3)
(9) to (4)
(7) to (5)
(10) to (7)
(9) to (10);

\path [draw,very thick] 
(4) to (5)
(11) to (12)
(12) to (13);

%%%%%%%%%%%%%%%%%%%%%%%%%%%%%%%%%%%%

\node[draw,fill,circle,inner sep=1pt] (16) at (-3,1){};

\node (116) at (-3.3,1){{\small $v_{\small 2}$}};
\node (216) at (-2.85,1.2){{\small $2$}};
\node[draw,fill,circle,inner sep=1pt] (17) at (-3,0) {};
\node (317) at (-2.75,.1){{\small $2$}};
\node[draw,fill,circle,inner sep=1pt] (18) at (-3,-1){};
\node (318) at (-2.85,-1.2){{\small $2$}};
\node[draw,fill,circle,inner sep=1pt] (19) at (-2,1){};
\node (119) at (-1.65,1){{\small $w_{\small 2}$}};
\node (219) at (-2.15,1.2){{\small $2$}};
\node[draw,fill,circle,inner sep=1pt] (20) at (-2,0) {};
\node (320) at (-2.25,.1){{\small $1$}};
\node[draw,fill,circle,inner sep=1pt] (21) at (-2,-1){};
\node (321) at (-2.15,-1.2){{\small $2$}};
\node[draw,fill,circle,inner sep=1pt] (15) at (-3.5,-2){};
\node (315) at (-3.15,-1.9){{\small $1$}};
\node[draw,fill,circle,inner sep=1pt] (14) at (-3.5,2){};
\node (314) at (-3.5,2.3){{\small $2$}};

\path [draw,postaction={on each segment={mid arrow=thick}}]
(19) to (16)
(20) to (19);

\path [draw,very thick,postaction={on each segment={mid arrow=very thick}}]
(16) to (14)
(16) to (17)
(18) to (21)
(11) to (14)
(17) to (12)
(13) to (15)
(18) to (15)
(17) to (18)
(21) to (20);

%%%%%%%%%%%%%%%%%%%%%%%%%%%%%%%%%%%%%%%%%%%%%%%

\node[draw,fill,circle,inner sep=1pt] (24) at (-1,1){};
\node (324) at (-.85,1.2){{\small $2$}};
\node[draw,fill,circle,inner sep=1pt] (25) at (-1,0) {};
\node (325) at (-.75,.1){{\small $2$}};
\node[draw,fill,circle,inner sep=1pt] (26) at (-1,-1){};
\node (326) at (-.85,-1.2){{\small $2$}};
\node[draw,fill,circle,inner sep=1pt] (27) at (0,1){};
\node (327) at (-.15,1.2){{\small $1$}};
\node[draw,fill,circle,inner sep=1pt] (28) at (0,0) {};
\node (328) at (-.25,.1){{\small $2$}};
\node[draw,fill,circle,inner sep=1pt] (29) at (0,-1){};
\node (129) at (.35,-1){{\small $w_{\small 5}$}};
\node (229) at (-.15,-1.2){{\small $1$}};
\node[draw,fill,circle,inner sep=1pt] (23) at (-1.5,-2){};
\node (323) at (-1.15,-1.9){{\small $1$}};
\node[draw,fill,circle,inner sep=1pt] (22) at (-1.5,2){};
\node (122) at (-1.15,2){{\small $v_{\small 3}$}};
\node (222) at (-1.5,2.3){{\small $2$}};

\path [draw,postaction={on each segment={mid arrow=thick}}]
(19) to (22)
(29) to (26)
(28) to (29);

\path [draw,very thick,postaction={on each segment={mid arrow=very thick}}]
(22) to (24)
(24) to (25)
(24) to (27)
(25) to (20)
(26) to (23)
(25) to (26);

\path [draw,very thick]
(23) to (21)
(27) to (28);

%%%%%%%%%%%%%%%%%%%%%%%%%%%%%%%%%%%%%%%%%%

\node[draw,fill,circle,inner sep=1pt] (32) at (1,1){};
\node (132) at (.75,1){{\small $v_{\small 4}$}};
\node (232) at (1.15,1.2){{\small $2$}};
\node[draw,fill,circle,inner sep=1pt] (33) at (1,0) {};
\node (333) at (1.25,.1){{\small $2$}};
\node[draw,fill,circle,inner sep=1pt] (34) at (1,-1){};
\node (334) at (1.15,-1.2){{\small $2$}};
\node[draw,fill,circle,inner sep=1pt] (35) at (2,1){};
\node (135) at (2.3,1.1){{\small $w_{\small 4}$}};
\node (235) at (2.25,0.8){{\small $2$}};
\node[draw,fill,circle,inner sep=1pt] (36) at (2,-1){};
\node (336) at (2.25,-.9){{\small $2$}};
\node[draw,fill,circle,inner sep=1pt] (31) at (.5,-2){};
\node (131) at (.5,-2.25){{\small $v_{\small 5}$}};
\node (231) at (.85,-1.9){{\small $2$}};
\node[draw,fill,circle,inner sep=1pt] (30) at (.5,2){};
\node (330) at (.5,2.3){{\small $2$}};

\path [draw,postaction={on each segment={mid arrow=thick}}]
(29) to (31)
(35) to (32)
(36) to (35);

\path [draw,very thick,postaction={on each segment={mid arrow=very thick}}]
(32) to (30)
(32) to (33)
(34) to (36)
(33) to (28)
(31) to (34);

\path [draw,very thick]
(27) to (30)
(33) to (34);

%%%%%%%%%%%%%%%%%%%%%%%%%%%%%
\end{tikzpicture}

\centering
\begin{tikzpicture}[dot/.style={draw,fill,circle,inner sep=1pt},scale=0.8]

%%%%%%%%%%%%%%%%%%%%%%%%%%%%%%%%%%%%%%%%%%%%%%%

\node[draw,fill,circle,inner sep=1pt] (1) at (-7,1){};
\node (101) at (-7.25,0.75){{\small $\mathbf{v_{\small 1}}$}};
\node (201) at (-7.25,1.1){{\small $1$}};
\node (T1) at (-7,-2.45){$\mathbf{T_1}$}; 
\node[draw,fill,circle,inner sep=1pt](2) at (-7,-1) {};
\node (102) at (-7.3,-.9){{\small $w_{\small 1}$}};
\node (202) at (-7.25,-1.2){{\small $2$}};

%%%%%%%%%%%%%%%%%%%%%%%%%%%%%%%%%%%%%%%%%

\node[draw,fill,circle,inner sep=1pt] (8) at (-5,1){};
\node (308) at (-4.8,1.3){{\small $2$}};
\node[draw,fill,circle,inner sep=1pt] (9) at (-5,0) {};
\node (309) at (-4.75,.1){{\small $2$}};
\node[draw,fill,circle,inner sep=1pt](5) at (-6,-1) {};
\node (305) at (-6.15,-1.2){{\small $1$}};
\node[draw,fill,circle,inner sep=1pt] (10) at (-5,-1){};
\node (310) at (-4.85,-1.25){{\small $2$}};
\node[draw,fill,circle,inner sep=1pt] (3) at (-6,1){};
\node (303) at (-6.25,1){{\small $2$}};
\node[draw,fill,circle,inner sep=1pt] (4) at (-6,0) {};
\node (304) at (-6.25,.1){{\small $2$}};
\node[draw,fill,circle,inner sep=1pt] (11) at (-4,1){};
\node (311) at (-4,1.3){{\small $1$}};\node[draw,fill,circle,inner sep=1pt] (12) at (-4,0) {};
\node (312) at (-4.25,.1){{\small $1$}};
\node[draw,fill,circle,inner sep=1pt] (13) at (-4,-1){};
\node (313) at (-4,-1.3){{\small $2$}};
\node[draw,fill,circle,inner sep=1pt] (7) at (-5.5,-2){};
\node (307) at (-5.15,-1.9){{\small $2$}};
\node[draw,fill,circle,inner sep=1pt] (6) at (-5.5,2){};
\node (306) at (-5.5,2.3){{\small $2$}};
\node (B1) at (-5.5,-2.45){$\mathbf{B_1}$};

\path [draw,very thick,postaction={on each segment={mid arrow=very thick}}]
(6) to (8)
(8) to (9)
(10) to (13)
(8) to (11)
(3) to (6)
(4) to (3)
(9) to (4)
(7) to (5)
(10) to (7)
(9) to (10);

\path [draw,dashed,postaction={on each segment={mid arrow=thick}}]
(1) to (2);

%%%%%%%%%%%%%%%%%%%%%%%%%%%%%%%%%%%%

\node[draw,fill,circle,inner sep=1pt] (16) at (-3,1){};
\node (116) at (-3.35,1){{\small $\mathbf{v_{\small 2}}$}};
\node (216) at (-2.8,1.25){{\small $2$}};
\node[draw,fill,circle,inner sep=1pt] (17) at (-3,0) {};
\node (317) at (-2.75,.1){{\small $2$}};
\node[draw,fill,circle,inner sep=1pt] (18) at (-3,-1){};
\node (318) at (-2.85,-1.25){{\small $2$}};
\node[draw,fill,circle,inner sep=1pt] (19) at (-2,1){};
\node (119) at (-2,.75){{\small $w_{\small 2}=w_{\small 3}$}};
\node (219) at (-2.15,1.25){{\small $2$}};
\node[draw,fill,circle,inner sep=1pt] (20) at (-2,0) {};
\node (320) at (-2.25,.1){{\small $1$}};
\node[draw,fill,circle,inner sep=1pt] (21) at (-2,-1){};
\node (321) at (-2,-1.3){{\small $2$}};
\node[draw,fill,circle,inner sep=1pt] (15) at (-3.5,-2){};
\node (315) at (-3.15,-1.9){{\small $1$}};
\node[draw,fill,circle,inner sep=1pt] (14) at (-3.5,2){};
\node (314) at (-3.5,2.3){{\small $2$}};
\node (T2) at (-3,-2.45){$\mathbf{T_2}$};

\path [draw,very thick,postaction={on each segment={mid arrow=very thick}}]
(16) to (14)
(16) to (17)
(18) to (21)
(17) to (12)
(18) to (15)
(17) to (18);

\path [draw,dashed,postaction={on each segment={mid arrow=thick}}]
(19) to (16);

%%%%%%%%%%%%%%%%%%%%%%%%%%%%%%%%%%%%%%%%%%%%%%%

\node[draw,fill,circle,inner sep=1pt] (24) at (-1,1){};
\node (324) at (-.8,1.3){{\small $2$}};
\node[draw,fill,circle,inner sep=1pt] (25) at (-1,0) {};
\node (325) at (-.75,.1){{\small $2$}};
\node[draw,fill,circle,inner sep=1pt] (26) at (-1,-1){};
\node (326) at (-.75,-1){{\small $2$}};
\node[draw,fill,circle,inner sep=1pt] (27) at (0,1){};
\node (327) at (0,1.3){{\small $1$}};
\node[draw,fill,circle,inner sep=1pt] (28) at (0,0) {};
\node (328) at (0,-.3){{\small $2$}};
\node[draw,fill,circle,inner sep=1pt] (29) at (0,-1){};
\node (129) at (.35,-1){{\small $w_{\small 5}$}};
\node (229) at (-.15,-1.2){{\small $1$}};
\node[draw,fill,circle,inner sep=1pt] (23) at (-1.5,-2){};
\node (323) at (-1.15,-1.9){{\small $1$}};
\node[draw,fill,circle,inner sep=1pt] (22) at (-1.5,2){};
\node (122) at (-1.83,2){{\small $\mathbf{v_{\small 3}}$}};
\node (222) at (-1.5,2.3){{\small $2$}};

\node (T3) at (-1,-2.45){$\mathbf{T_3}$};

\path [draw,very thick,postaction={on each segment={mid arrow=very thick}}]
(22) to (24)
(24) to (25)
(24) to (27)
(25) to (20)
(26) to (23)
(25) to (26);

\path [draw,dashed,postaction={on each segment={mid arrow=thick}}]
(19) to (22);

%%%%%%%%%%%%%%%%%%%%%%%%%%%%%%%%%%%%%%%%%%

\node[draw,fill,circle,inner sep=1pt] (32) at (1,1){};
\node (132) at (.7,1){{\small $\mathbf{v_{\small 4}}$}};
\node (232) at (1.2,1.3){{\small $2$}};
\node[draw,fill,circle,inner sep=1pt] (33) at (1,0) {};
\node (333) at (1,-.3){{\small $2$}};
\node[draw,fill,circle,inner sep=1pt] (34) at (1,-1){};
\node (334) at (1,-.7){{\small $2$}};
\node[draw,fill,circle,inner sep=1pt] (35) at (2,1){};
\node (135) at (2.3,1.1){{\small $w_{\small 4}$}};
\node (235) at (2.25,0.8){{\small $2$}};
\node[draw,fill,circle,inner sep=1pt] (36) at (2,-1){};
\node (336) at (2,-.7){{\small $2$}};
\node[draw,fill,circle,inner sep=1pt] (31) at (.5,-2){};
\node (131) at (.5,-2.25){{\small $\mathbf{v_{\small 5}}$}};
\node (231) at (.85,-1.9){{\small $2$}};
\node[draw,fill,circle,inner sep=1pt] (30) at (.5,2){};
\node (330) at (.5,2.3){{\small $2$}};
\node (T4) at (1.75,1.65){$\mathbf{T_4}$};
\node (T5) at (1.25,-2.45){$\mathbf{T_5}$};

\path [draw,very thick,postaction={on each segment={mid arrow=very thick}}]
(32) to (30)
(32) to (33)
(34) to (36)
(33) to (28)
(31) to (34);

\path [draw,dashed,postaction={on each segment={mid arrow=thick}}]
(29) to (31)
(35) to (32);

%%%%%%%%%%%%%%%%%%%%%%%%%%%%%
\end{tikzpicture}
\end{minipage}
\hspace{0.5cm} 
\begin{minipage}[c]{4.8cm}
We take the weighted oriented subgraph $K$ of $D_1$ induced by $V(D_1)\setminus \{ w_1,w_2,w_4,w_5\}$. In the following figure,  $T_1,\ldots ,T_5$ are the ROT's and $B_1$ is the unicycle oriented subgraph such that the parent of $T_i$ is $v_i$ and $W^{T_i}=\{ w_i\}$ for $i=1,\ldots ,5$. Furthermore, $H=\cup_{i=1}^{5} \ T_i\ \cup B_1$ is a $\star$-semi-forest with $W_{1}^{H}=\{ w_1,w_5\}$ $W_{2}^{H}=\{ w_2,w_4\}$ and $V(H)=V(K)$. Therefore, $H$ is a ge\-ne\-ra\-ting $\star$-semi-forest of $K$.
\end{minipage}

\begin{example}\rm
Let $D_2$ be an oriented weighted graph such as the Figure. Let $K^{4}=H=D_2[x_1,x_2,x_3]$ be an induced weighted oriented subgraph of $D_2$ and $H$ a generating $\star$-semi-forest of $K^{4}$ with $\tilde{H}=H=\{ x_1,x_2,x_3\} \subseteq V^{+}$ and $W_1=W_2=\emptyset$. Then, by Theorem \ref{theorem-oct29}, there is a strong vertex cover $\mathcal{C}$ of $D_2$, such that $V(K^{4})\subseteq \mathcal{C}$. For this purpose, first we take a minimal vertex cover $\mathcal{C}_i$ of $D$ and define $\mathcal{C}_{i}^{\prime}=(\mathcal{C}_i\setminus W_1)\cup N_D(W_1)\cup N_{D}^{+}(W_2\cup \tilde{H})$. Finally, we use the algorithm in the proof of the Proposition \ref{GeneratinAStrongVC}.

\noindent
In this example, $\mathcal{C}_{i}^{\prime}=\mathcal{C}_i\cup N_{D}^{+}(\tilde{H})=\mathcal{C}_i\cup N_{D}^{+}(\{ x_1,x_2,x_3\} )=\mathcal{C}_i\cup \{ x_1,x_2,x_3,y_1,y_2\}$.
\end{example}
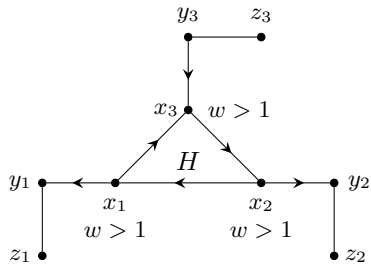
\begin{minipage}[c]{4.7cm} 
\centering
\begin{tikzpicture}[dot/.style={draw,fill,circle,inner sep=1pt},scale=.97]
%%%%%%%%%%%%%%%%%%%%%%%%%%%%%%%%%%%%%%%%%

\node[draw,fill,circle,inner sep=1pt] (2) at (-2,-7){};
\node (12) at (-2.3,-7){{\small $z_{\tiny 1}$}};
\node[draw,fill,circle,inner sep=1pt] (3) at (2,-7){};
\node (13) at (2.3,-7){{\small $z_{\tiny 2}$}};
\node[draw,fill,circle,inner sep=1pt] (4) at (1,-4){};
\node (14) at (1,-3.7){{\small $z_{\tiny 3}$}};
\node[draw,fill,circle,inner sep=1pt] (5) at (-2,-6){};
\node (15) at (-2.3,-6){{\small $y_{\tiny 1}$}};
\node[draw,fill,circle,inner sep=1pt] (6) at (-1,-6){};
\node (16) at (-1,-6.3){{\small $x_{\tiny 1}$}};
\node (116) at (-1,-6.65){{\small $w>1$}};
\node[draw,fill,circle,inner sep=1pt] (7) at (2,-6){};
\node (17) at (2.35,-6){{\small $y_{\tiny 2}$}};
\node[draw,fill,circle,inner sep=1pt] (8) at (1,-6){};
\node (18) at (1,-6.3){{\small $x_{\tiny 2}$}};
\node (118) at (1,-6.65){{\small $w>1$}};
\node[draw,fill,circle,inner sep=1pt] (9) at (0,-4){};
\node (19) at (0,-3.7){{\small $y_{\tiny 3}$}};
\node[draw,fill,circle,inner sep=1pt] (10) at (0,-5){};
\node (20) at (-.3,-5){{\small $x_{\tiny 3}$}};
\node (120) at (0.7,-5){{\small $w>1$}};
\node (0) at (0,-5.7){$H$};

\path [draw,postaction={on each segment={mid arrow=thick}}] 
(6) to (10)
(10) to (8)
(8) to (6)
(6) to (5)
(8) to (7)
(9) to (10);

\draw[-] (2)-- (5);
\draw[-] (3)--(7);
\draw[-] (4)--(9);

%%%%%%%%%%%%%%%%%%%%%%%%%%%%%%%%%%%%%%%%%%

\end{tikzpicture}
\end{minipage}
\begin{minipage}[c]{8.6cm}
\centering

\begin{itemize}[noitemsep]
\item If we take $\mathcal{C}_1=\{ x_1,x_2,y_1,y_2,y_3\}$ as a minimal vertex cover of $D_2$, then $\mathcal{C}_{1}^{\prime}=\{ x_1,x_2,x_3,y_1,y_2,y_3\}$ and $L_3(\mathcal{C}_{1}^{\prime})\setminus N_{D}^{+}(\tilde{H}) =\emptyset$. Thus, by Proposition \ref{GeneratinAStrongVC}, $\mathcal{C}=\mathcal{C}_{1}^{\prime}=\{ x_1,x_2,x_3,y_1,y_2,y_3\}$ is a strong vertex cover such that $V(K^{4})\subseteq \mathcal{C}$ and it is no minimal. Furthermore, in this case it is enough to know the orientation of $E(K^{4})$, since $L_3(\mathcal{C})=\{ x_1,x_2,x_3\}$ and $N_{D_2}(x_i)\subseteq \mathcal{C}$ for $i=1,2,3$.
\item Now, if we take $\mathcal{C}_2=\{ x_1,x_2,x_3,z_1,z_2,z_3\}$ as a minimal vertex cover of $D_2$, then $\mathcal{C}_{2}^{\prime}=V(D_2)\setminus \{ y_3\}$ and $L_3(\mathcal{C}_{2}^{\prime})\setminus N_{D}^{+}(\tilde{H}) =\{ z_1,z_2\}$. Thus, by the algorithm in the proof of Proposition \ref{GeneratinAStrongVC}, $\mathcal{C}^{\prime}=\mathcal{C}_{4}^{\prime}=\{ x_1,x_2,x_3,y_1,y_2,z_3\}$ is a strong vertex cover such that $V(K^{4})\subseteq \mathcal{C}^{\prime}$. \medskip
\end{itemize}
\end{minipage}

\noindent
Since $\mathcal{C}$ and $\mathcal{C}^{\prime}$ are no minimal with $L_3(\mathcal{C})=\{ x_1,x_2,x_3\}$ and $L_3(\mathcal{C}^{\prime})=\{ x_1,x_2\}$, then $D_2$ is mixed. Furthermore, $V(D_2)$ has a partition in complete graphs: $K^{i}=D_2[y_i,z_i]$ for $i=1,2,3$ and $K^{4}=D_2[x_1,x_2,x_3]$ .  

\begin{example}\rm
Let $D_3=(G,\mathcal{O},w)$ be the following oriented graph. Hence, 
\end{example}
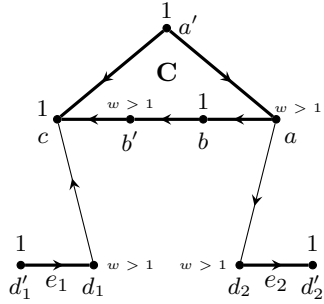
\begin{minipage}[c]{4.7cm} 
\centering
\begin{tikzpicture}[dot/.style={draw,fill,circle,inner sep=1pt},scale=.97]
%%%%%%%%%%%%%%%%%%%%%%%%%%%%%%%%%%%%%%%%%

\node[draw,fill,circle,inner sep=1pt] (2) at (-5,.5){};
\node (12) at (-5,.2){{\small $b^{\prime}$}};
\node (22) at (-5,.7){{\tiny $w>1$}};
\node[draw,fill,circle,inner sep=1pt] (5) at (-6.5,-1.5) {};
\node (15) at (-6.5,-1.8){{\small $d_{\small 1}^{\prime}$}};
\node (25) at (-6.5,-1.2){{\small $\small 1$}};
\node[draw,fill,circle,inner sep=1pt](3) at (-4,.5) {};
\node (13) at (-4,.2){{\small $b$}};
\node (23) at (-4,.75){{\small $1$}};
\node[draw,fill,circle,inner sep=1pt] (6) at (-2.5,-1.5){};
\node (16) at (-2.5,-1.8){{\small $d_{\small 2}^{\prime}$}};
\node (26) at (-2.5,-1.2){{\small $\small 1$}};
\node[draw,fill,circle,inner sep=1pt] (1) at (-6,.5){};
\node (11) at (-6.2,.25){{\small $c$}};
\node (21) at (-6.2,.65){{\small $1$}};
\node[draw,fill,circle,inner sep=1pt] (4) at (-5.5,-1.5) {};
\node (14) at (-5.5,-1.8){{\small $d_{\small 1}$}};
\node (24) at (-5,-1.5){{\tiny $w>1$}};
\node[draw,fill,circle,inner sep=1pt] (7) at (-3,.5){};
\node (11) at (-2.8,.25){{\small $a$}};
\node (21) at (-2.7,.65){{\tiny $w>1$}};
\node[draw,fill,circle,inner sep=1pt] (8) at (-3.5,-1.5){};
\node (18) at (-3.5,-1.8){{\small $d_{\small 2}$}};
\node (28) at (-4,-1.5){{\tiny $w>1$}};
\node[draw,fill,circle,inner sep=1pt] (9) at (-4.5,1.75){};
\node (19) at (-4.2,1.8){{\small $a^{\prime}$}};
\node (29) at (-4.5,2){{\small $1$}};
\node (C) at (-4.5,1.15){$\mathbf{C}$};
\node (e1) at (-6,-1.75){$e_1$};
\node (e2) at (-3,-1.75){$e_2$};

\path [draw,very thick,postaction={on each segment={mid arrow=thick}}] 
(9) to (7)     (7) to (3)     (3) to (2)     (2) to (1)     
(5) to (4)     (8) to (6)     (9) to (1);

\path [draw,postaction={on each segment={mid arrow=thick}}]
(7) to (8)     (4) to (1);

%%%%%%%%%%%%%%%%%%%%%%%%%%%%%%%%%%%%%%%%%%
\end{tikzpicture}
\end{minipage}
\begin{minipage}[c]{8.6cm}
\centering
\begin{itemize}[noitemsep]
\item $G$ has no $3$- and $4$-cycles. Moreover, $girth(G)=5$.
\item $G$ is an $SCQ$-graph with $S_G=\{ e_1,e_2\}$, $C_G=\{ C\}$ and $Q_G=\emptyset$, where $e_1=\{ d_1,d_{1}^{\prime}\}$, $e_2=\{ d_2,d_{2}^{\prime}\}$ and  $C=(a^{\prime},a,b,b^{\prime},c,a^{\prime})$.
\item $C$ has the $\star$-property. But $e_2$ has a generating
$\star$-semi-forest. Hence, by Theorem \ref{No3,4,5Cyc-Unmix}, $I(D)$ is mixed.
\end{itemize}
\end{minipage}

\begin{example}\rm
Let $D_4=(G,\mathcal{O},w)$ be the following oriented weighted graph. Hence,
\end{example}
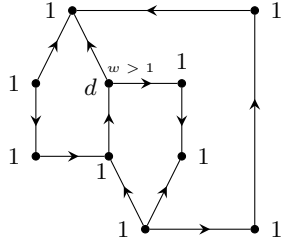
\begin{minipage}[c]{4cm} 
\centering
\begin{tikzpicture}[dot/.style={draw,fill,circle,inner sep=1pt},scale=.97]
%%%%%%%%%%%%%%%%%%%%%%%%%%%%%%%%%%%%%%%%%

\node[draw,fill,circle,inner sep=1pt] (2) at (-5,.5){};
\node (12) at (-4.7,.7){{\tiny $w>1$}};
\node (22) at (-5.25,.45){{\small $d$}};
\node[draw,fill,circle,inner sep=1pt] (5) at (-5,-.5) {};
\node (15) at (-5.1,-.7){{\small $1$}};
\node[draw,fill,circle,inner sep=1pt](3) at (-4,.5) {};
\node (13) at (-4,.8){{\small $1$}};
\node[draw,fill,circle,inner sep=1pt] (6) at (-4,-.5){};
\node (16) at (-3.7,-.5){{\small $1$}};
\node[draw,fill,circle,inner sep=1pt] (1) at (-6,.5){};
\node (11) at (-6.3,.5){{\small $1$}};
\node[draw,fill,circle,inner sep=1pt] (4) at (-6,-.5) {};
\node (14) at (-6.3,-.5){{\small $1$}};
\node[draw,fill,circle,inner sep=1pt] (7) at (-3,1.5){};
\node (17) at (-2.7,1.5){{\small $1$}};
\node[draw,fill,circle,inner sep=1pt] (8) at (-3,-1.5){};
\node (18) at (-2.7,-1.5){{\small $1$}};
\node[draw,fill,circle,inner sep=1pt] (10) at (-4.5,-1.5){};
\node (20) at (-4.8,-1.5){{\small $1$}};
\node[draw,fill,circle,inner sep=1pt] (9) at (-5.5,1.5){};
\node (19) at (-5.8,1.5){{\small $1$}};

\path [draw,postaction={on each segment={mid arrow=thick}}]
(2) to (3)
(2) to (9)
(5) to (2)
(10) to (8)
(8) to (7)
(7)to (9)
(1) to (9)
(1) to (4)
(4) to (5)
(10) to (5)
(10) to (6)
(3) to (6);

%%%%%%%%%%%%%%%%%%%%%%%%%%%%%%%%%%%%%%%%%%
\end{tikzpicture}
\end{minipage}
\begin{minipage}[c]{9.3cm}
\centering
\begin{itemize}[noitemsep]
\item $G=P_{10}$, then by Theorem \ref{wellcovered-characterization2} and Remark \ref{1star}, $G$ is well-covered and $I(G)$ is unmixed.
\item $d$ is not a sink and $d\in V^{+}$.
\item By Proposition \ref{P10-Unmix}, $I(D_4)$ is unmixed. 
\end{itemize}
\end{minipage}

\bibliographystyle{amsplain}

\end{document}